\documentclass[a4paper]{amsart}
\usepackage[toc,title,page]{appendix}
\usepackage{amscd,amsthm,amsfonts,amssymb,amsmath,esint}
\usepackage{amsmath,tikz-cd}
\usepackage{float}
\restylefloat{table}
\usepackage{fullpage}
\usepackage[all]{xy}
\usepackage{mathtools}
\usepackage{cite}
\usepackage[colorlinks=false,
linkcolor=red,       
anchorcolor=red,  
citecolor=blue,        
]{hyperref}

\newcommand{\BA}{{\mathbb {A}}} 
\newcommand{\BC}{{\mathbb {C}}} 
 
\newcommand{\BG}{{\mathbb {G}}} \newcommand{\BH}{{\mathbb {H}}}

\newcommand{\BO}{{\mathbb {O}}} \newcommand{\BP}{{\mathbb {P}}}
\newcommand{\BQ}{{\mathbb {Q}}} \newcommand{\BR}{{\mathbb {R}}}
\newcommand{\BS}{{\mathbb {S}}}

 \newcommand{\BZ}{{\mathbb {Z}}}

\newcommand{\CA}{{\mathcal {A}}} 
\newcommand{\CC}{{\mathcal {C}}} 
\newcommand{\CE}{{\mathcal {E}}} \newcommand{\CF}{{\mathcal {F}}}
 \newcommand{\CH}{{\mathcal {H}}}
 
 \newcommand{\CL}{{\mathcal {L}}}
\newcommand{\CM}{{\mathcal {M}}} 
\newcommand{\CO}{{\mathcal {O}}} \newcommand{\CP}{{\mathcal {P}}}
\newcommand{\CQ}{{\mathcal {Q}}} \newcommand{\CR}{{\mathcal {R}}}
 \newcommand{\CT}{{\mathcal {T}}}
\newcommand{\CU}{{\mathcal {U}}} \newcommand{\CV}{{\mathcal {V}}}

\newcommand{\fg}{{\mathfrak{g}}} 
 
 \newcommand{\fl}{{\mathfrak{l}}}
 
\newcommand{\fo}{{\mathfrak{o}}} 
 
\newcommand{\fs}{{\mathfrak{s}}} \newcommand{\ft}{{\mathfrak{t}}}

\theoremstyle{plain}
\newtheorem{thm}{Theorem}[section] 
\newtheorem*{main thm}{Main Theorem} 
\newtheorem{cor}[thm]{Corollary}
\newtheorem{lem}[thm]{Lemma}

\newtheorem{ques}[thm]{Question} \newtheorem{prop}[thm]{Proposition}

\newtheorem{defn}[thm]{Definition}
 \newtheorem{lem-defn}[thm]{Lemma-Definition}
\newtheorem*{Notations}{Notations}

\newtheorem*{Main Problem}{Main Problem}

\theoremstyle{remark} \newtheorem{remark}[thm]{Remark}
\theoremstyle{remark} 
\theoremstyle{remark} \newtheorem{example}[thm]{Example}

\usepackage[skip=0.5\baselineskip]{caption}

\numberwithin{equation}{section}

\everymath{\displaystyle}
\author{Yingqi Liu}
\title{Moduli of codimension two linear sections of subadjoint varieties}
\address{Yingqi Liu,  AMSS, Chinese Academy of Sciences, 55 ZhongGuanCun East Road, Beijing, 100190, China and  University of Chinese Academy of Sciences, Beijing, China}
\email{liuyingqi@amss.ac.cn}
\begin{document}

	\maketitle
	\begin{abstract}
	 	Let $G$ be a simple algebraic group of type $F_{4}$, $E_{6}$, $E_{7}$ or $E_{8}$, and let $\mathfrak{g}$ be its Lie algebra. The adjoint variety $X_{ad} \subseteq \mathbb{P} \mathfrak{g}$ is defined as the unique closed orbit of the adjoint action of $G$ on $\mathbb{P}\mathfrak{g}$.  $X_{ad}$ is a Fano contact manifold covered by lines in $\BP \fg$. 
	 	The subadjoint variety $S \subseteq \BP W$ is denoted by the variety of lines on $X_{ad}$ through a fixed point $x$, where $W \subseteq  T_{x}X$ is taken as the contact hyperplane. It follows from a result in representation theory of Vinberg that the GIT quotient space of codimension two linear sections of $S$ is isomorphic to the weighted projective space $\BP(1,3,4,6)$. In this note, we investigate the problem of finding a geometric interpretation of the above isomorphism. As a main result, for each $\fg$ of the above type, we construct a natural open embedding of the GIT quotient space of nonsingular codimension two linear sections of $S$ into $\BP(1,3,4,6)$ whose complement is a fixed hypersurface of degree 12. The key ingredient of our construction is to apply a correspondence of Bahargava and Ho which relates the above moduli problem to a moduli problem on curves of genus one.
		\end{abstract}
	
	\section{Introduction}
	Let $G$ be a simply connected simple algebraic group over $\BC$, and let $\mathfrak{g}$ be its Lie algebra. The adjoint variety $X_{ad} \subseteq \mathbb{P} \mathfrak{g}$ is defined as the unique closed $G$-orbit of the adjoint action on $\mathbb{P}\mathfrak{g}$. $X_{ad}$ is a Fano cantact manifold, and it is  covered by lines when $G$ is not of type $A$ or $C$. In this case, the subadjoint variety $S \subseteq \mathbb{P}W$ is denoted by the variety of lines on $X_{ad}$ through a fixed point $x$, where we identify a line through $x$ with its tangent direction at $x$, and $W \subseteq T_{x}X$ is taken as the contact hyperplane. $S$ is a rational homogenous space under the action of the isotropy subgroup of $G$. In fact let $H$ be the semi-simple part of the isotropy subgroup of $G$ at $x$. Then $W$ is an irreducible subspace under the isotropy action of $H$ on $T_{x}X$. And $S$ is the unique closed orbit of the action of $H$ on $\mathbb{P}W$. The contact structure on $X_{ad}$ gives rise to a $H$-invariant nondegenerate skew bilinear form $w$ on $W$, with respect to which $S \subseteq \BP W$  is known as a Legendrian variety. The importance of understanding the geometry of subadjoint varieties (especially as Legendrian varieties) stems from the celebrated conjecture by Lebrun and Salamon, which asserts that adjoint varieties are the only examples of Fano contact manifolds. We refer to \cite{LM07} for more details.\par
	Classical  examples of subadjoint varieties are the twisted cubic in $\BP^{3}$, and the products $\BP^{1} \times \BQ^{n}$ of a line and a smooth hyperquadric of dimension $n \geqslant 1$, which correspond to $\fg=\fg_{2}$ and $\fg=\fs \fo_{n+6}$ respectively. In this article we focus on the following five subadjoint varieties, where we use $H_{G},W_{G}$ and $S_{G}$ to express the objects associated with a fixed group $G$.\par 
	\begin{table}[h!]
		\centering
		\begin{tabular}{|c|c|c|c|c|}
			\hline
			$\text{Type of\,\,\,}\mathfrak{g}$ & $H_{G} $ & Rep.$W_{G}$ &  $S_{G} \subseteq \mathbb{P}W_{G} $ & $dim(S_{G})$\\
			\hline 	
			$D_{4}$ & $SL_{2} \times SL_{2} \times SL_{2}$ & $\mathbb{C}^{2} \otimes \mathbb{C}^{2} \otimes \mathbb{C}^{2}$ & $\mathbb{P}^{1} \times \mathbb{P}^{1}\times \mathbb{P}^{1}, Segre$ & 3\\
			$F_{4}$ & $SP_{6}$ & $ \wedge^{3}_{0} \mathbb{C}^{6}$ &  $Lag(3,6)$, Pl{\"u}cker  & 6\\
			$E_{6}$ & $SL_{6}$ & $ \wedge^{3} \mathbb{C}^{6}$ &  $Gr(3,6)$, Pl{\"u}cker & 9\\
			$E_{7}$ & $Spin_{12}$ & half spin  & $\BS_{6}$, minimal &15 \\
			$E_{8}$ & $E_{7}$ & minuscule  &  $E_{7}/P_{7}$, minimal & 27\\ \hline 
		\end{tabular} 
		\caption{}
		\label{tab:my_label}
	\end{table}
The subadjoint varieties listed in Table 1, together with the twisted cubic, are also known as the varieties in the third row of $geometric$ $Freudenthal$-$Tits$ $magic$ $rectangle$ \cite[Section 6]{Ma13}. They can be constructed in a uniform way as twisted cubic ``curves'' over Hermitian cubic Jordan algebras (see \ref{bir_map}). 
 In this note we focus on the study of their linear sections.
 It is known that there are exactly four orbits of the action of $H$ on $\mathbb{P}W$ (see Theorem \ref{proj_geom}). It follows that all nonsingular hyperplane sections are isomorphic to each other. Moreover when $\fg$ is not of type $D_{4}$ it is proved in \cite{CFL} that these nonsingular hyperplane sections are rigid under smooth deformation. \par 
	The main goal of this note is to describe the GIT quotient space of codimension two linear sections of the subadjoint variety $S \subseteq \mathbb{P}W$ listed in Table 1. 
 Under the identification $W \cong W^{\vee}$ via $w$, for a plane $[U] \in Gr(2,W)$, we denote $\BP U^{\perp} \subseteq \BP W$ to be the linear subspace isotropic to $U$, and denote $S_{U}=\BP U^{\perp} \cap S$ to be the associated linear section. We are going to study the GIT quotient of the $H$-action on the pair $(Gr(2,W)),\mathcal{O}_{Gr}(1))$, where $\mathcal{O}_{Gr}(1)$ is the Pl{\"u}cker line bundle on $Gr(2,W)$, and we fix the linearization of $H$ on it provided by the linear action of $H$ on $W$. Recall that the GIT quotient is given by $Gr(2,W)^{ss}//H=Proj(R_{G})$, 
	 where $R_{G}=\oplus_{k \geqslant 0} H^{0}(Gr(2,W),\mathcal{O}_{Gr}(k))^{H}$ is
	the ring of $H$-invariant sections. Note that $R_{G}$ can be naturally identified with the invariant ring $\mathbb{C}[\mathbb{C}^{2} \otimes W]^{SL_{2} \times H}$, under which an invariant section of degree $d$ is identified with an invariant function of degree $2d$. A key motivation of our study is the following description of $R_{G}$ arising from Vinberg's theory of $\theta$-groups.
		\begin{thm}\cite{Vi}\label{mot_}
	Let $G$ be a simple algebraic group. The invariant ring $\mathbb{C}[\mathbb{C}^{2} \otimes W]^{SL_{2} \times H}$ is a polynomial ring. If $G$ is of type $F_{4},E_{6},E_{7}$ or $E_{8}$, the invariant ring $\mathbb{C}[\mathbb{C}^{2} \otimes W]^{SL_{2} \times H}$ is a polynomial ring in four generators, of degree $2,6,8,12$ respectively. In particular  $Proj(R_{G}) \cong \mathbb{P}(1,3,4,6)$ is a weighted projective space of dimension 3. 
	\end{thm}
	\begin{remark}
		If $G$ is of type $D_{4}$,  then the invariant ring $\mathbb{C}[\mathbb{C}^{2} \otimes W]^{SL_{2} \times H}$ is a polynomial ring in four generators, of degree 2,4,4,6 respectively.
		If we additionally consider the action of $S_{3}$ on $W=\BC^{2} \otimes \BC^{2} \otimes \BC^{2} $ given by permutations of factors, then the ring $\mathbb{C}[\mathbb{C}^{2} \otimes W]^{SL_{2} \times H \times S_{3}}$ becomes also a polynomial ring in four generators, of degree $2,6,8,12$ respectively. We will explain these facts in the Appendix, see (A.5).
	\end{remark}
It is remarkable that $Proj(R_{G})$ is independent of $G$ when $G$ is of type $F_{4},E_{6},E_{7}$ or $E_{8}$. This motivates us to investigate the following problem which seeks for a geometric interpretation.
	\begin{Main Problem}\label{mp}
		Assume $G=F_{4},E_{6},E_{7}$ or $E_{8}$, then:\par 
		1. Provide a geometric interpretation for the generators of $R_{G}$, and use it to characterize the semi-stable locus $Gr(2,W_{G})^{ss}$.\par 
		2. Construct geometric identifications between the GIT quotients $Gr(2,W_{G})^{ss}//H_{G}$ when $G$ varies.
	\end{Main Problem}
Following Remark 1.2, our main idea to tackle the problem is to compare it to the case of $G=D_{4}$. On one direction, there exists an embedding of representations arsing from the construction of the representation $(H,W)$  in \cite{LM02} using composition algebras:
\begin{prop}(Corollary \ref{S3})
		Let $G$ be of type $F_{4},E_{6},E_{7}$ or $E_{8}$. Then there exists an embedding of representations $(Lie(H_{D_{4}}),W_{D_{4}}) \subseteq (Lie(H_{G}),W_{G})$. Moreover let $\overline{H}_{G} \subseteq PGL(W_{G})$ be the image of $H_{G}$ in $PGL(W_{G})$. Then there is a subgroup $S_{3} \subseteq \overline{H}_{G}$ which acts on $\mathbb{P}W_{D_{4}}=\mathbb{P} (\mathbb{C}^{2} \otimes \mathbb{C}^{2} \otimes \mathbb{C}^{2}) $ by permutations of factors.
\end{prop}
For the converse, to associate an $H_{G}$-equivalent class of planes in $W_{G}$ to an $H_{D_{4}} \times S_{3}$-equivalent class of planes in $W_{D_{4}}$, we apply the following correspondence constructed by Bhargava and Ho in \cite{BH}.
	\begin{defn}
		The tangent variety $TS$ of $S \subseteq \mathbb{P}W$ is a quartic hypersurface (Theorem \ref{proj_geom}).	A plane $U \subseteq W$ is called nondegenerate if $TS$ intersects $\mathbb{P}U$ at four distinct points. For a tensor $v \in \mathbb{C}^{2} \otimes W$, we say it is $nondegenerate$ if it defines a nondegenerate plane in $W$.
	\end{defn}
	\begin{thm}\cite[Theorem 6.28]{BH}\label{cr}
		There is a one to one correspondence between nondegenerate $GL_{2} \times H$-orbits on $\mathbb{C}^{2} \otimes W$ and isomorphic classes of triples $(C,L,\kappa: C \rightarrow S)$, where $C$ is a nonsingular projective curve of genus one, $L$ is a line bundle on $C$ of degree 2, and $\kappa$ is a morphism from $C$ to $S$ such that $\kappa^{*}(\mathcal{O}_{S}(1)) \cong L^{\otimes 3}$. 
	\end{thm} 
The correspondence is built upon  some special projective properties of subadjoint varieties (see Section 2.3). The case where $G$ is of type $D_{4}$ can be written more concretely as follows. In particular it provides a geometric interpretation for the generators of the invariant ring in this case (see Proposition A.2). 
	\begin{prop}\cite[Section 2.4, Theorem 6.4]{BH} \label{D_{4}'}
		In Theorem \ref{cr}, if $G=D_{4}$, then:\par 
		(1) A plane $U$ is nondegenerate if and only if the linear section $S_{U}$ is nonsingular.\par 
		(2) The triple $(C,L,\kappa: C \rightarrow \mathbb{P}^{1} \times \mathbb{P}^{1} \times \mathbb{P}^{1})$ is the same as the datum $(C,L,L_{i}, i=1,2,3)$, where $(L_{1},L_{2},L_{3})$ is a tuple of three  line bundles of degree 2 on $C$, such that $3L=L_{1}+L_{2} +L_{3}$, and $ L \not= L_{i}$, for any $i=1,2,3$. \par 
		(3) There exists a choice of normalization for the  generators of $SL_{2} \times H$-invariants $a_{2},a_{4}.a_{4}',a_{6}$, such that given a nondegenerate tensor in $\BC^{2} \otimes W$ corresponding to the datum $(C,L,L_{i}, i=1,2,3)$ as in (2), the Jacobian of $C$ may be given in normal form as 
		\begin{equation}
			y^{2}=x^{3}+a_{8}x+a_{12},
			\end{equation}
		on which the points $P_{1}\doteq L-L_{1}=(a_{4},a_{6}),P_{2} \doteq L-L_{2}=(a_{4}',a_{6}'),P_{3} \doteq L-L_{3}=(a_{4}'',a_{6}'')$ lie, such that $a_{2}$ is the slope of the line connecting $P_{1}$ and $P_{2}$, and $a_{6}'=a_{6}+a_{2}(a_{4}'-a_{4})$.
	\end{prop}
The key ingredient of our main result is the following, which relates nondegenerate planes in $W_{G}$ to  nondegenerate planes in $W_{D_{4}}$.
\begin{thm}\label{3_lines}
	Let $G=D_{4}, F_{4},E_{6},E_{7}$ or $E_{8}$, then:\par 
	(1). A plane $U \subseteq W_{G}$ is nondegenerate if and only if $S_{U}$ is nonsingular. There is a $H$-invariant section $\sigma_{12} \in H^{0}(Gr(2,W),\mathcal{L}^{\otimes 12})^{H}$, such that $[U] \in Gr(2,W)$ is nondegenerate if and only if $\sigma_{12}(U) \not =0$.\par 
	(2).(Proposition \ref{3_line}) Let  $U \subseteq W_{G}$ be a nondegenerate plane, and let $(C,L,\kappa:C \rightarrow S_{G})$ be the associated datum defined in Theorem \ref{cr}. Then there exists (uniquely) three line bundles $L_{1},L_{2},L_{3}$ on $C$ of degree two satsifying:\par 
(2.1) $3L=L_{1}+L_{2}+L_{3}$, and $L \not= L_{i}$ for $i=1,2,3$;\par 
(2.2) For any two distinct points $p,q$ on $C$, $p+q \in |L_{i}|$ for some $i$ if and only if their images $\kappa(p),\kappa(q)$ are isotropic under $w$, or equivalently $\kappa(p)$ and $\kappa(q)$ can be connected by at most two lines on $S_{G}$.
\end{thm}
 Denote the open subset parametrizing nonsingular sections by $\CU_{G}=\{\sigma_{12} \not=0\} \subseteq Gr(2,W_{G})$. As a nondegenerate plane $U$ varies in $\CU_{G}$, we show that Theorem \ref{3_lines}(2) induces a family of three line bundles (Proposition \ref{family}). This motivates us to define
 the moduli functor $\CM$ (Definition \ref{mod_func}) which parametrizes the datum $(C,L,L_{1},L_{2},L_{3})$, where $C$ is a curve of genus one, $L$ is a degree two line bundle on $C$, and $\{L_{1},L_{2},L_{3}\}$ is a multiset of three  degree two line bundles on $C$ satisfying the condition as in Theorem 1.7(2.1). \par 
 
 It follows from Theorem 1.7(1) that $\CU_{G} \subseteq Gr(2,W_{G})^{ss}$. We denote $M_{G}= \CU_{G} // H_{G} $ to be its GIT quotient. In the case where $G=D_{4}$, we denote $M=M_{D_{4}}//S_{3}$ to be the GIT quotient by $S_{3}$. 
We deduce the following from Proposition 1.6:
\begin{prop}\label{D_{4}}
	(1). For the action of $H_{D_{4}}=SL_{2} \times SL_{2} \times SL_{2}$ on $(Gr(2,W_{D_{4}}),\CO_{Gr}(1))$, a plane $[U]$ is GIT-stable if and only if $[U]$ is nondegenerate. \par   
	(2). There is a natural transformation $\xi: \CM \rightarrow h_{M}$, making it a coarse moduli space of $\CM$. \par 
    (3). The parametrization in Prop 1.6 (3) induces a natural open embedding of $M$ into $\mathbb{P}(1,3,4,6)$ whose complement is a hypersurface of degree 12.
\end{prop}
	Combined with Proposition 1.6 and Proposition 1.8, our main theorem provides an answer to our Main Problem for nonsingular sections:
	\begin{main thm}\label{Main}
		For $G=F_{4},E_{6},E_{7}$ or $E_{8}$, Theorem \ref{3_lines} induces a natural datum $\alpha_{G} \in \CM(\CU_{G})$, which by Proposition \ref{D_{4}}(2) gives a  morphism $\xi(\CU_{G})(\alpha_{G}): \CU_{G} \rightarrow M$. This morphism is $H_{G}$-invariant, and the induced quotient morphism $\Phi_{G}:M_{G} \rightarrow M$ is an isomorphism.
	\end{main thm}
\begin{remark}
	$\Phi_{G}$ induces a birational map: $Gr(2,W_{G})^{ss}//H_{G} \dashrightarrow Gr(2,W_{D_{4}})^{ss}//H_{D_{4}} \times S_{3}$. We conjecture that it can be extended to an isomorphism, which would provide a complete answer to the Main Problem (2). 
\end{remark}
To study the moduli of codimension two linear sections of $S_{G}$ one step further, we ask the following question which has an affirmative answer in the case where $G=F_{4}$ (\cite[Proposition 4]{DM22}).
\begin{ques}
Let $G$ be of type $F_{4},E_{6},E_{7}$ or $E_{8}$. Let $[U],[U'] \in \CU_{G}$ be two general nondegenerate planes in $W_{G}$, and suppose that $\phi: S_{U} \rightarrow S_{U'}$ is an isomorphism. Dose there exist $h \in \overline{H}_{G}$ such that $[U']=h.[U]$ and $\phi=h^{*}$?
\end{ques}
 
The article is organized as follows. In Section 2 we first recall some basic facts on subadjoint varieties. Then we will prove Proposition 1.3, Theorem \ref{3_lines}(1) and  Proposition  \ref{D_{4}}(1). Moreover we shall construct a natural morphism $\sigma_{G}: M \rightarrow M_{G}$ which will be proved to be a section of $\Phi_{G}$ defined in the Main Theorem. In Section 3 we will prove Theroem 1.7(2) and show that it induces a natural datum $\alpha_{G} \in \CM(\CU_{G})$. In Section 4 we first verify Proposition 1.8(2), then we finish the proof of our Main Theorem. In the appendix we will prove Proposition 1.8(3).
	\begin{Notations}
		1.All schemes are assumed to be locally noetherian. A variety is defined as an integral separated scheme of finite type over $\mathbb{C}$.
		Given a line bundle $L$ on a variety $X$ and a nonzero section $\sigma \in H^{0}(X,L)$, we denote the associated effective Cartier divisor by $[\sigma]$. \par 
		2. Given a scheme $S$, a scheme $X$ over $S$ is said to be projective over $S$ if $X$ is a closed subscheme of $ \mathbb{P}_{S}(\CE)$ for some finite rank vector bundle $\CE$ on $S$.  For a projective and flat family $X$ over $S$, we denote $Div_{X/S}$ to be the scheme representing the functor of families of relative effective divisors on $X/S$.  
	\end{Notations}
\section{Preliminaries}
\subsection{Representation $(H,W)$}
To study the projective geometry of subadjoint varieties in details, we review two perspectives to recover the representation $(H,W)$. 
\subsubsection{Cubic Jordan algebras}
The first one is given from cubic Jordan algebras, from which we can express each subadjoint variety as a cubic twisted ``curve" over a cubic Jordan algebra (\ref{bir_map}).
\begin{defn}
	A $Jordan$ $algebra$ $J$ over $\mathbb{C}$ is a vector space over $\mathbb{C}$ with a bilinear product $\bullet$ satisfying the following axioms: \par 
	\begin{equation*}
		x \bullet y =y \bullet x;\,\, x^{2} \bullet (x \bullet y)=x \bullet (x^{2} \bullet y), \,\,\, \text{\,for $x,y \in J$}.
	\end{equation*}
\end{defn}
A Jordan algebra is called simple if it has no proper ideals, and it is called semi-simple if it is the direct sum of simple Jordan algebras.
A cubic Jordan algebra is a Jordan algebra in which every element satisfies a cubic equation. In the following, we always assume $J$ to be a semi-simple cubic Jordan algebra.\par 

The cubic equation naturally defines an admissible cubic form $N$ on $J$(see \cite[Definition 2]{Cl}). We denote its linearization to be $N: J \times J \times J \rightarrow \mathbb{C}$, with $N(x)=N(x,x,x)$. The cubic equation  also defines a  trace form $Tr(\, , \, )$ on $J$. It is nondegenerate when $J$ is semi-simple and thus one can define a quadratic adjoint map $\#$ on $J$ satisfying $Tr(x^{\#},y)=N(x,x,y)$ for all $y \in J$. The following table lists the Jordan algebras associated to Table 1 (see Theorem 2.4 below), where we remark that the last four columns exactly correspond to simple cubic Jordan algebras.\par 
\begin{table}[H]\label{T2}
	\centering
	\begin{tabular}{|c|c|c|c|c|c|}
		\hline
		$\mathfrak{g}$ &  $D_{4}$ & $F_{4}$ &	$E_{6}$ &       $E_{7}$ &	 $E_{8}$\\ \hline 
		$J$ &  $\CH_{3}(0)$ &$\CH_{3}(\mathbb{C})$	& $\CH_{3}(\mathbb{C}\oplus \mathbb{C})$ &
		$\CH_{3}(\mathbb{H}_{\BC})$  &  $\CH_{3}(\mathbb{O}_{\BC})$\\ \hline
	\end{tabular} 
	\caption{}
	\label{tab:my_label}
\end{table}
Here for a composition algebra $\BA$, $\mathcal{H}_{3}(\BA)$ is the $Hermitian$ cubic Jordan algebra over $\BA$ defined as follows:
\begin{equation*}
	\mathcal{H}_{3}(\BA)=\Bigg\{ \left(
	\begin{matrix}
		\xi_{1} & a & b \\
		\overline{a} & \xi_{2} & c \\
		\overline{b} & \overline{c} & \xi_{3} 
	\end{matrix}
	\right)
	, \xi_{1},\xi_{2},\xi_{3} \in \BC, a,b,c \in \BA.
	\Bigg\}
\end{equation*}
Recall that there are exatly four complex composition algebras: $\BA=\BC,\BC \oplus \BC, \BH_{\BC}, \BO_{\BC}$, which are the complexifications of $\BR,\BC,\BH,\BO$ respectively. $\CH_{3}(0) \subseteq \CH_{3}(\BA)$ is defined as the subalgebra of diagonal matrices. For the definition of $N$ and $\#$ on $\CH_{3}(\BA)$, see for example \cite[Example 05]{Kr}.\par
Let $J$ be a semi-simple cubic Jordam algebra, we define the group 
$$ NP(J)=\{g \in GL(J)| N(gA)=N(A), \forall A \in J\}.$$
To describe the action of $NP(J)$ on $J$, we recall the following definition.
\begin{defn}
Let $J$ be a semi-simple cubic Jordan algerbra and let $N$ be the cubic norm on $J$. The rank of an element $x \in J$ is defined by the following relations:
\begin{align*}
\textit{rank}\,\, x=3 \,\, &\textit{if} \,\, N(x) \not=0;\\
	\textit{rank}\,\, x=2 \,\, &\textit{if} \,\, N(x)=0  \,\,\textit{and} \,\,x^{\#} \not=0;\\
	\textit{rank}\,\, x=1 \,\, &\textit{if} \,\, x^{\#}=0  \,\,\textit{and} \,\,x \not=0;\\
	\textit{rank}\,\, x=0 \,\, &\textit{if} \,\, x=0.
\end{align*}
\end{defn}
\begin{prop}\cite[Proposition 12]{Kr}
Let $\BA$ be a complex composition algebra. Let $J=\CH_{3}(\BA)$ be the Hermitian cubic Jordan algebra over $\BA$. Then the rank is invariant under the action of $NP(J)$.  And $NP(J)$ acts transitively on the sets of elements of rank 0,1,2 respectively. In the case of rank $3$, the group acts transitively on the elements of a given norm $k \in \BC^{*}$. 
\end{prop}
Given a semi-simple cubic Jordan algebra $J$, the set of $Zorn$ $matrices$ over $J$ is defined as:
\begin{equation*}
 \mathcal{C}_{J}=\Bigg\{ \left(
	\begin{matrix}
		 a & b \\
	c & d
	\end{matrix}
	\right)
	, a,d \in \mathbb{C},\, b,c \in J.
	\Bigg\}
\end{equation*}
The determinant of a Zorn matrix $A=(a,b,c,d)$ is given by: 
\begin{equation}
	det(A)=(ad-Tr(b,c))^{2}-4Tr(b^{\#},c^{\#})+4aN(c)+4dN(b). 
\end{equation}
\par 
We then define $D$ to be the symmetric quadrilinear map on $\mathcal{C}_{J}$ with $D(A,A,A,A)=det(A)$. One can also define a natural skew bilinear form $w$ on $\mathcal{C}_{J}$ defined by:
\begin{equation}\label{w}
	w(A,A')=ad'-a'd+Tr(b',c)-Tr(b,c').
\end{equation}
This form is nondegenerate as the trace form $Tr( \,,)$ on $J$ is nondegenerate. One then can define the cubic adjoint map $\flat$ on $\mathcal{C}_{J}$ by the properties $w(A^{\flat},B)=D(A,A,A,B)$ and $(A^{\flat})^{\flat}=-det(A)^{2}A$. Let $\widetilde{H}_{J}$ be the group of linear automorphism of $\mathcal{C}_{J}$ preserving both det and $w$. Then the dictionary reads as follows:
\begin{thm}\cite[Proposition 4.4]{Cl} and \cite[Table 1]{Kr}.\label{dict_01}
	Up to a finite cover or a finite subgroup, the irreducible representations $(H_{G},W_{G})$ are the same as the representations $(\widetilde{H}_{J},\mathcal{C}_{J})$, where $G$ ranges over simple groups which are not of type $A$ or $C$ and $J$ ranges over semi-simple cubic Jordan algebras.
\end{thm}
For a detailed proof see \cite[Section 4]{Cl}. We give some remarks here for further discussions.  
\begin{remark}\label{top_G}
	Given a cubic Jordan algebra $J$ associated to a simple group $G$ as above, denote the connected component of $\widetilde{H}_{J}$ containing the identity by  $(\widetilde{H}_{J})^{o}$ and denote the image of $(\widetilde{H}_{J})^{o}$ in $PGL(\CC_{J})$ by $\overline{H}_{J}$.
	Then under the identification $W_{G} \cong \CC_{J}$, $H_{G}$ is a finite cover of $(\widetilde{H}_{J})^{o}$ and hence a finite cover of $\overline{H}_{J}$. The action of $H_{G}$ on $Gr(2,W_{G})$ is induced by the action of $\overline{H}_{J}$ on $Gr(2,\CC_{J})$. 
\end{remark}

\begin{example}
If $J=\CH_{3}(0)$, then $\CC_{J}=\{A=(a,(b_{1},b_{2},b_{3}),(c_{1},c_{2},c_{3}),d)): a,b_{i},c_{j},d \in \BC\}$.  The forms $D$ and $w$ are given by:	
\begin{align*}
	det(A)=&a^{2}d^{2} + b_{1}^{2}c_{1}^{2}+ b_{2}^{2}c_{2}^{2}+ b_{3}^{2}c_{3}^{2}+4(ac_{1}c_{2}c_{3}+b_{1}b_{2}b_{3}d)\\ &-2(ab_{1}c_{1}d+ab_{2}c_{2}d+ab_{3}c_{3}d+b_{1}b_{2}c_{1}c_{2}+b_{2}b_{3}c_{2}c_{3}+b_{1}b_{3}c_{1}c_{3}).\\
	w(A,A')=&ad'-(b_{1}c_{1}'+b_{2}c_{2}'+b_{3}c_{3}')+(b_{1}'c_{1}+b_{2}'c_{2}+b_{3}'c_{3})-a'd.
\end{align*}
 And we can identify  $\CC_{J}$ with $\BC^{2} \otimes \BC^{2} \otimes \BC^{2}$ as follows. Take $e_{1},e_{2}$ to be a basis of $\BC^{2}$ and write $e_{ijk}=e_{i}\otimes e_{j} \otimes e_{k}$, then we identify:
 \begin{align*}
 	(0,(1,0,0),(0,0,0),0) &\leftrightarrow e_{211}, (0,(0,1,0),(0,0,0),0) \leftrightarrow e_{121}, (0,(0,0,1),(0,0,0),0) \leftrightarrow e_{112}, \\
 	(0,(0,0,0),(1,0,0),0) &\leftrightarrow e_{122}, (0,(0,0,0),(0,1,0),0) \leftrightarrow e_{212}, (0,(0,0,0),(0,0,1),0) \leftrightarrow e_{221}, \\  
  (1,(0,0,0),(0,0,0),0) &\leftrightarrow e_{111}, (0,(0,0,0),(0,0,0),1)\leftrightarrow e_{222}.
  \end{align*}
\end{example}
Under the correspondence $(Lie(H_{G}),W_{G}) \cong (Lie(\overline{H}_{J}),\CC_{J})$, the subadjoint variety $S=S_{G}$ is realized as a compactification of the algebra $J$. More precisely, $S$ is the proper image of the following birational map:
\begin{align}\label{bir_map}
	Q_{J}:  \mathbb{P}(\mathbb{C} \oplus J) &\dashrightarrow \mathbb{P}(\mathbb{C} \oplus J \oplus J \oplus \mathbb{C})=\mathbb{P}\CC_{J}\\
	[t:\alpha] &\rightarrow [t^{3}:t^{2}\alpha:t\alpha^{\#}:N(\alpha)] \nonumber
\end{align}
In particular we can identify $J$ with an open subset of $S$  as follows:
\begin{align}
	q_{J}: J &\rightarrow S,\\
	\alpha &\rightarrow [1:\alpha:\alpha^{\#}:N(\alpha)]. \nonumber
\end{align}
This identification induces a natural $NP(J)$-action on $S$. We describe this action geometrically as follows.
Denote $x_{0}=Q_{J}([1:0])$, we naturally identify $T_{x_{0}}S$ with $J$ through  $q_{J}$. Denote the stabilizer subgroup of the action of $\overline{H}_{J}$ on $S$ at $x_{0}$ by $Stab_{x_{0}}(\overline{H}_{J})$. It is a parabolic subgroup which is the product of a vector group and a Levi subgroup $L_{J}$.  $L_{J}$ is the product of a one parameter subgroup $\mathbb{G}_{m}$ and a semi-simple subgroup $M_{J}$, where the action of $\mathbb{G}_{m}$ on $\CC_{J}$ is given as follows:
\begin{align}\label{dis_c*}
	\mathbb{C}^{*} \times \CC_{J} &\rightarrow \CC_{J} \\
	(t, (a,b,c,d)) &\rightarrow (t^{3}a,tb,t^{-1}c,t^{-3}d).\nonumber
\end{align}
And the isotropic action of $M_{J}$ on $T_{x_{0}}(S) \cong J$ preserves the cubic form on $J$, which induces an isogeny of $M_{J}$ onto $NP(J)^{o}$, namely the connected component of $NP(J)$ containing the identity. In particular from Proposition 2.3 we have the following fact:
\begin{cor}
Let $G$ be of type $F_{4},E_{6},E_{7}$ or $E_{8}$. Denote the subadjoint variety associated to $G$ by $S_{G} \subseteq \BP W_{G}$, and let $x$ be a point on $S_{G}$. Then the isotropy action of the stabilizer subgroup $Stab_{x_{0}}(H_{G})$ on $\BP T_{x}S_{G}$ has three orbits.
\end{cor}
\begin{defn}
Let $x$ be a point on the subadjoint variety $S_{G}$. A tangent direction $[w] \in \BP T_{x}S_{G}$ is said to be generic if $[w]$ lies in the open orbit of the action of $Stab_{x}(H_{G})$ on $\BP T_{x}S_{G}$.	
\end{defn}
\subsubsection{Triality}
The second one is given from the following construction of exceptional Lie algebras as well as their special representations in \cite{LM02} using composition algebras which highlights the $triality$ principle. This construction enables us to define an $S_{3}$-action on $\BP W_{G}$.
 \begin{defn}
 	Let $\BA$ be a complex composition algebra, denote the triality group by
 	\begin{equation*}
T(\BA)=\{\theta=(\theta_{1},\theta_{2},\theta_{3}) \in SO(\BA)^{3},\theta_{3}(xy)=\theta_{1}(x)\theta_{2}(y), \forall x,y \in \BA \}.
 	\end{equation*}
 	\end{defn}
It is known that  $T(\BA)$ is an algebraic subgroup of $SO(\BA)^{3}$.
  There are three actions of $T(\BA)$ on $\BA$ corresponding to its three projections on $SO(\BA)$, and we denote these representations by $\BA_{1},\BA_{2},\BA_{3}$. Denote by $t(\BA)$ the Lie algebra of the triality group $T(\BA)$. Let $\BA$ and $\BA'$ be two composition algebras, and let
  \begin{equation*}
   \mathfrak{g}(\BA,\BA')=\mathfrak{t}(\BA) \times \mathfrak{t}(\BA') \oplus (\BA_{1} \otimes \BA'_{1}) \oplus (\BA_{2} \otimes \BA'_{2})  \oplus (\BA_{3} \otimes \BA'_{3}).
   \end{equation*}
Then $\mathfrak{g}(\BA,\BA')$ has a natural structure of Lie algebra, which is semi-simple and whose type is given by Freudental's magic square(\cite[Theorem 2.1]{LM02}).  If $\BA'=\BH_{\BC}$, then $\mathfrak{g}(\BA,\BH_{\BC})= \mathfrak{sp}_{6}, \mathfrak{sl}_{6},  \mathfrak{so}_{12},$ or $\mathfrak{e}_{7}$, which  corresponds to $Lie(H)$ of the last four rows in Table 1. And  we can identify $\mathfrak{t}(\BH_{\BC})$ with $\mathfrak{sl}_{2} \times \mathfrak{sl}_{2} \times \mathfrak{sl}_{2}$ by(cf. \cite[Lemma 1]{BS}):
 \begin{equation*}
 	\mathfrak{sl}_{2} \times \mathfrak{sl}_{2} \times \mathfrak{sl}_{2} \rightarrow \mathfrak{t}(\BH_{\BC}) : \,\,\,\,
 	(x,y,z) \rightarrow (l_{y}-r_{z},l_{z}-r_{x},l_{y}-r_{x}),
 \end{equation*}
 where we identify $\mathfrak{sl}_{2}=Im(\BH_{\BC})$ and $l_{x}, r_{x}$ denote the operators of left and right multiplication by $x$, respectively. We take the Cartan decomposition of $\fs\fl_{2}=Im(\BH_{\BC})$ as $Im(\BH_{\BC})=\langle h \rangle \oplus \langle x \rangle \oplus\langle y \rangle$, where 
 \[
h=\begin{pmatrix}
1 & 0 \\
0 & -1
\end{pmatrix},
x=\begin{pmatrix}
	0 & 1 \\
	0 & 0
\end{pmatrix},
y=\begin{pmatrix}
	0 & 0 \\
	1 & 0
\end{pmatrix}.
 \]
We denote by $\alpha \in \langle h \rangle ^{\vee}$ the simple root satisfying $\alpha(h)=2$. And let $\alpha_{1},\alpha_{2},\alpha_{3}$ be the respective simple roots of $\ft(\BH_{\BC})$ corresponding to each factor. Denote by $w_{1},w_{2},w_{3}$ the corresponding fundamental dominant wights of $\ft(\BH_{\BC})$. Let $U_{1},U_{2},U_{3}$ be the natural 2-dimensional representations of our three copies of $\mathfrak{sl}_{2}$, then $$(\BH_{\BC})_{1}=U_{2}\otimes U_{3},(\BH_{\BC})_{2}=U_{3}\otimes U_{1}, (\BH_{\BC})_{3}=U_{1}\otimes U_{2}.$$ \par 
The representation $(H,W)$ can be recovered as follows:
 \begin{thm}\cite[Theorem 4.1]{LM02}\label{dict_02}
 	Let $V_{\BA}=\BA_{1} \otimes U_{1} \oplus \BA_{2} \otimes U_{2} \oplus \BA_{3} \otimes U_{3} \oplus U_{1} \otimes U_{2} \otimes U_{3}$. Then:\par 
 	(i) $V_{\BA}$ has a natural structure of $\mathfrak{g}(\BA,\BH_{\BC})$-module, which is simple and isomorphic to $(Lie(\widetilde{H}_{\mathcal{H}_{3}(\BA)}),\CC_{\mathcal{H}_{3}({\BA})})$. \par 
 	(ii)	The action of the subalgebra $\mathfrak{t}(\BH_{\BC}) \cong \mathfrak{sl}_{2} \times \mathfrak{sl}_{2} \times \mathfrak{sl}_{2}$ on $V_{\BA}$ is given by:
 	\begin{equation*}
 		(x_{1},x_{2},x_{3}).(a_{i} \otimes u_{i})=a_{i} \otimes x_{i}.u_{i}, \,\, \text{for}\,\, i=1,2,3 \,\, \text{and} \,\,  (x_{1},x_{2},x_{3}).(u_{1} \otimes u_{2} \otimes u_{3})=x_{1}.u_{1} \otimes x_{2}.u_{2} \otimes x_{3}.u_{3}, 
 	\end{equation*}
 In particular the sub-representation $(\mathfrak{t}(\BH_{\BC}),U_{1} \otimes U_{2} \otimes U_{3})$ is isomorphic to $(Lie(H_{D_{4}}),W_{D_{4}})$. \par 
 (iii) A highest weight vector of the module $(\mathfrak{t}(\BH),U_{1} \otimes U_{2} \otimes U_{3})$ is a highest weight vector of the module $(\mathfrak{g}(\BA,\BH_{\BC}),V_{\BA})$.
\end{thm}
Then Proposition 1.3 follows from the following corollary.
\begin{cor}\label{S3}
	Under the identification in Theorem \ref{dict_02}, there is a subgroup  $S_{3} \subset \overline{H}_{\CH_{3}(\BA)}$ which acts on $\mathbb{P} (U_{1} \otimes U_{2} \otimes U_{3})$ by permutations of the factors.
	\end{cor}
Here we first give the construction of the subgroup and verify some properties. We will conclude its proof after introducing some projective properties of subadjoint varieties.\par 
 For any composition algebra $\BA$, there is an $S_{3}$-action on $\mathfrak{t}(\BA)$ generated by the following two automorphisms(cf. \cite[Proposition 3.6.4, Lemma 3.5.9]{SV00}):
\begin{equation*}
\sigma(t_{1},t_{2},t_{3})=(t_{2},\widehat{t_{3}},\widehat{t_{1}}),\,\,\, \tau(t_{1},t_{2},t_{3})=(t_{3},\widehat{t_{2}},t_{1}),
\end{equation*}
	where $\widehat{t}(x):= \overline{t(\overline{x})}$ for any $x \in \BA$. Now given two composition algebras $\BA,\BA'$, we can extend the $S_{3}$-action on $\mathfrak{t}(\BA) \times \mathfrak{t}(\BA')$ to an $S_{3}$-action on $\mathfrak{g}(\BA,\BA')$ generated by:
	\begin{align*}
	\sigma: \,\, &\BA_{1} \otimes \BA'_{1} \rightarrow 	\BA_{3} \otimes \BA'_{3},\,\,\,
		\BA_{3} \otimes \BA'_{3} \rightarrow 	\BA_{2} \otimes \BA'_{2},\,\,\,
			\BA_{2} \otimes \BA'_{2} \rightarrow 	\BA_{1} \otimes \BA'_{1}. \\		
		&a_{1} \otimes a'_{1} \rightarrow \overline{a_{1}} \otimes \overline{a'_{1}},\,\,\,a_{3} \otimes a'_{3} \rightarrow \overline{a_{3}} \otimes \overline{a'_{3}},\,\,\, a_{2} \otimes a'_{2} \rightarrow a_{2} \otimes a'_{2}.
	\end{align*}
	and 
		\begin{align*}
		\tau: \,\, &\BA_{1} \otimes \BA'_{1} \rightarrow 	\BA_{3} \otimes \BA'_{3},\,\,\,
		\BA_{3} \otimes \BA'_{3} \rightarrow 	\BA_{1} \otimes \BA'_{1},\,\,\,
		\BA_{2} \otimes \BA'_{2} \rightarrow 	\BA_{2} \otimes \BA'_{2}. \\		
		&a_{1} \otimes a'_{1} \rightarrow a_{1} \otimes a'_{1},\,\,\,a_{3} \otimes a'_{3} \rightarrow a_{3} \otimes a'_{3},\,\,\, a_{2} \otimes a'_{2} \rightarrow -\overline{a_{2}} \otimes \overline{a'_{2}}.
	\end{align*}
The verification that this defines an action on the Lie algebra is just a computation. Then one easily checks that under this identification the induced $S_{3}$-action on $\mathfrak{sl}_{2} \times \mathfrak{sl}_{2} \times \mathfrak{sl}_{2}$ is generated by: 
\begin{equation}
\sigma(x,y,z)=(y,z,x) \,\,\text{and}\,\, \tau(x,y,z)=(z,y,x).  
\end{equation}
Note that for each composition algebra $\BA$, $V_{\BA}$ is the unique irreducible $\mathfrak{g}(\BA,\BH)$-module of dimension $dim(V_{\BA})$ up to isomorphisms. Thus the $S_{3}$-action on $\mathfrak{g}(\BA,\BH)$ induces a homomorphism $f: S_{3} \rightarrow PGL(V_{\BA})$ such that for any $g \in S_{3}$ and any representative $F_{g}$ of $f(g)$ in $GL(V_{\BA})$ we have:
\begin{equation}\label{S_{3}_equation}
 F_{g}(X.v)=(g.X).(F_{g}(v)),\,\, \text{for any} \,\, X \in \mathfrak{g}(\BA,\BH)\,\,\text{and}\,\, v \in V_{\BA}.
 \end{equation}
By Theorem \ref{dict_02} the highest weight of the action of $t(\BH_{\BC})$ on $V_{\BA}$ is $w_{1}+w_{2}+w_{3}$ which corresponds to its action on $U_{1} \otimes U_{2} \otimes U_{3}$. Thus from (2.6)(2.7) we deduce that the action of $S_{3}$ on $\BP V_{\BA}$ stabilizes the subspace generated by a maximal weight vector in $U_{1} \otimes U_{2} \otimes U_{3}$, and hence it stabilizes $\BP(U_{1} \otimes U_{2} \otimes U_{3})$. Furthermore from (2.6)(2.7) one can also easily see that the action of $S_{3}$ on $\BP(U_{1} \otimes U_{2} \otimes U_{3})$ is given by permutation of factors. \par 
\subsection{Projective geometry of subadjoint varieties}
  In this subsection, we fix a simple group $G$ of type $D_{4},F_{4},E_{6},E_{7}$ or $E_{8}$.  We consider the subadjoint variety $S \subseteq \mathbb{P}W$ associated to $G$.  We identify $W$ with $\CC_{J}$ where $J$ is the cubic Jordan algebra associated to $G$.  We denote $Sec(S)$ to be the secant variety of $S$ in $\mathbb{P}W$, and we denote $TS$ to be the tangential variety of $S$ in $\mathbb{P}W$. We also identify $\mathbb{P}W$ with its dual via the skew form $w$.  We summarize some projective properties of subadjoint varieties as follows.
\begin{thm}\cite[Section 9]{Cl}\label{proj_geom}
	(1). $Sec(S)=\mathbb{P}W$, $TS=\{det=0\}$, $Sing(TS)=\{\flat=0 \}$.\par 
	(2). For any point $x \in S$, the embedded tangent space $\widehat{T_{x}S}$ is a maximal isotropic subspace in $W$ with respect to the skew form $w$. For any point $y \in \mathbb{P}W$, denote by $H_{y}$ the hyperplane defined by $y$. Then $H_{y}$ is tangent at $x$ if and only if $y \in \mathbb{P} \widehat{T_{x}S}$. In other words $TS$ is identified to the dual variety $S^{*}$.\par 
	(3). We have the following special properties:\par 
	(3.1) Each point in $\mathbb{P}W \backslash TS$ lies on a unique secant line.\par 
	(3.2) For a point $\alpha=[A]$ in $TS \backslash Sing(TS)$, the point $x_{\alpha}=[A^{\flat}]$  lies in $S$, and $\alpha$  lies on a unique tangent line given by the join of $\alpha$ and $x_{\alpha}$.\par
	(3.3) Each point in  $Sing(TS)$ belongs to an infinite number of tangent lines.\par 
	(4). The action of $H$ on $\mathbb{P}W$ has exactly four orbits, namely $\mathbb{P}W \backslash TS$, $TS \backslash Sing(TS)$, $Sing(TS)\backslash S$, and $S$.
\end{thm}
Next we review the geometry of lines on subadjoint varieties. For any two points $p,q$ on $S$, denote by $l(p,q)$ the minimal number of lines on $S$ connecting $p$ and $q$.
\begin{thm}\label{lines_S}
	For any $p,q \in S$, we have $l(p,q) \leqslant 3$. Moreover:\par 
	(i) Two pairs of points $(p,q)$ and $(p',q')$ on $S$ are conjugate under the $H$-action if and only if $l(p,q)=l(p',q')$.\par 
	(ii) For a pair of points $p=[A],q=[B]$ on $S$, then $l(p,q) \leqslant 2$ if and only if $p,q \in \mathbb{P} \widehat{T_{z}S}$ for some $z \in S$, and if and only if $w(A,B)=0$. For a point $z$ on $S$, $\mathbb{P}\widehat{T_{z}S} \cap S$ equals to the union of the lines on $S$ through $z$.\par 
	(iii) For any $x \in \mathbb{P}W$,
	$x \in \mathbb{P}W \backslash TS$ if and only if $x$ lies on a secant line $l_{pq}$ of $S$, such that $l(p,q)=3$; and $x \in TS\backslash Sing(TS)$ if and only if it lies on a tangent line $l_{p}$ of $S$ at some point $p \in S$, such that $[T_{p}l_{p}]$ is generic.
\end{thm}
As a first corollary we can finish the proof of Corollary 2.11. We first verify the following:
\begin{lem}
	For the natural embedding $(Lie(H_{D_{4}}),W_{D_{4}}) \subset (Lie(H_{G}),W_{G})$ given in Theorem \ref{dict_02}, the determinant form $D_{W_{G}}$ and the skew form $w_{W_{G}}$ on $W_{G}$  restrict to the associated forms $D_{W_{4}}$ and $w_{W_{D_{4}}}$  on $W_{D_{4}}$. 
\end{lem}
\begin{proof}
	The fact that $w_{W_{G}}|_{W_{D_{4}}}=w_{W_{D_{4}}}$ follows from the remark following \cite[Theorem 4.1]{LM02}.
	Denote by $H^{sc}_{G}$ the simply-connected cover of $H_{G}$.  Then we have a natural embedding $SL_{2} \times SL_{2} \times SL_{2} \subseteq H^{sc}_{G}$. And by Theorem \ref{dict_02}(iii) there is an equivariant embedding of subadjoint varieties $(S_{D_{4}},\BP W_{D_{4}}) \subseteq (S_{G},\BP W_{G})$. Thus from Theorem \ref{proj_geom}(1) and (4), to show $D_{W_{G}}|_{W_{D_{4}}}=D_{W_{D_{4}}}$, it suffices to show that $TS_{G} \cap \BP W_{D_{4}}=TS_{D_{4}}$ or equivalently $TS_{G}$ does not contain $\BP W_{D_{4}}$. In fact take any two points $p=[A],q=[B]$ on $S_{D_{4}}$ with $l_{S_{D_{4}}}(p,q)=3$.
	Then from Theorem \ref{lines_S}(ii) we have $w_{W_{D_{4}}}(A,B)=w_{W_{G}}(A,B) \not=0$. Thus $l_{S_{G}}(p,q)=3$ and from Theorem \ref{lines_S}(iii) we have $[A + B] \in \BP W_{D_{4}} \backslash TS_{G}$.
\end{proof}
\begin{proof}[Proof of Corollary 2.11]
	It suffices to verify that the image of our morphism $f: S_{3} \rightarrow PGL(V_{\BA})$ lies in $\overline{H}_{\CH_{3}(\BA)}$. 
	Recall that for any $g \in S_{3}$, we have shown that $f(g)$ stabilizes the weight subspace $V_{w_{1}+w_{2}+w_{3}}$  of weight $w_{1}+w_{2}+w_{3}$. Thus we can choose a representative $F_{g} \in GL(V_{\BA})$ of $f(g)$  fixing  $V_{w_{1}+w_{2}+w_{3}}$. Now we verify our claim in three steps.\par 
	We first prove that $F_{g}$ fixes the determinant form $D_{W_{G}}$. Note that the action of $S_{3}$ on $\mathfrak{g}(\BA,\BH)$ naturally induces an $S_{3}$-action on $H^{sc}_{G}$. And from (2.7) we have:
	\begin{equation}\label{S_{3}_equation}
		F_{g}(h.v)=(g.h).(F_{g}(v)),\,\, \text{for any} \,\, h \in H^{sc}_{G}\,\,\text{and}\,\, v \in V_{\BA}.
	\end{equation}
In particular this implies that $F_{g}$ stabilizes the subadjoint variety $S_{G}=H^{sc}_{G}.[V_{w_{1}+w_{2}+w_{3}}]$. Thus it stabilizes $T S_{G}$ and hence  by Theorem 2.12 it acts on $D$ by a scalar. On the other hand note that $F_{g}$ acts on $U_{1} \otimes U_{2} \otimes U_{3}$ by a certain permutation of factors.   From Example 2.6 one easily checks that it preserves the determinant form on $W_{D_{4}}$. Then by the above lemma we conclude that $F_{g}$ preserves the form $D_{W_{G}}$. \par 
	Next we check that it preserves the skew form $w_{W_{G}}$. As $w_{W_{G}}$ is a $H_{G}^{sc}$-invariant form,  $F_{g}(w_{W_{G}})$ is also an $H_{G}^{sc}$-invariant form by (2.8). Then from Theorem \ref{lines_S}(i) and by the above lemma it suffices to prove that $F_{g}$ preserves the skew form on $W_{D_{4}}$, which again can be easily checked from Example 2.6.\par
	Now if $G$ is not of type $E_{6}$, then $\widetilde{H}_{J}$ is already connected by \cite[Proposition 16.(ii)]{Kr} and thus $f(g) \in \overline{H}_{\CH_{3}(\BA)}$.  If $G$ is of type $E_{6}$, then the corresponding representation is $(H^{sc}_{G},W_{G})\cong (SL_{6},\wedge^{3}\BC^{6})$ and one can easily check that the natural embedding is given through identifying $\BC^{6}=U_{1}\oplus U_{2} \oplus U_{3}$. Then as $F_{g}$ acts on $U_{1} \otimes U_{2} \otimes U_{3}$ by a certain permutation, one deduces that the action of $F_{g}$ on $\wedge^{3}\BC^{6}$ comes from an element in $SL_{6}$ and hence $f(g) \in \overline{H}_{\CH_{3}(\BA)}$. 
\end{proof}
We can now prove Theorem \ref{3_lines}(1). 
\begin{proof}[Proof of Theorem \ref{3_lines}(1)]
	The $H$-invariant section $\sigma_{12}$ comes from the determinant form $D$. More precisely we consider the following $SL_{2} \times H$-invariant function:
	\begin{equation}
		W \times W \rightarrow \mathbb{C}:\,\,\, (\alpha,\beta) \rightarrow disc_{a,b}(det(a\alpha+b\beta)),
	\end{equation}
where for a fixed pair $(\alpha, \beta) \in 	W \times W$, $disc_{a,b}(det(a\alpha+b\beta))$ is the discriminant of the quartic binary form $det(a\alpha +b\beta)$ over $a,b$. Note that $det$ is $H$-invariant and the discriminant of a binary quartic form is known as an $SL_{2}$-invariant of degree 12. Thus the above function is an $SL_{2} \times H$-invariant of degree 24. It defines a $H$-invariant section $\sigma_{12} \in H^{0}(Gr(2,W),\CO_{Gr(2,W)}(12))$ such that for any plane $[U] \in Gr(2,W)$, $\sigma_{12}([U]) \not =0$ if and only if $[U]$ is nondegenerate.\par 
Given a plane $[U] \in Gr(2,W)$, for any point $x \in TS \cap \BP U$, we denote $S_{x}=\{p \in S:  x \in \BP \widehat{T_{p}S}\}$.
Then by Theorem 2.12(2): 
\begin{align*}
Sing(S_{U})=\{p \in S_{U}:\,\, \text{there exists some}\,\, x \in \BP U\,\, \text{such that}\,\, H_{x}\,\, \text{is tangent at}\,\, p\}
=(\bigcup_{x \in \BP U \cap TS} S_{x}) \cap S_{U}.
\end{align*}
If $[U]$ is nondegenerate, then for any point $\alpha=[A] \in  \BP U$, we claim that $w(A^{\flat},-)$ is not identically zero on $U$. Assume otherwise, then for any point $\beta=[B] \in \BP U$, we have $w(A^{\flat},B)=D(A,A,A,B)=0$. But then as $TS=\{det=0\}=\{[B]: D(B,B,B,B)=0\}$, we conclude that $TS$ intersects $\BP U$ at $\alpha=[A]$ of order at least two, which is a contradiction.
 Now assume $S_{U}$ is singular, then there exists a point $\alpha=[A] \in TS \cap \BP U$ such that $S_{\alpha} \cap S_{U} \not= \varnothing$. By the above discussion we have $A^{b} \not=0$, and hence by Theorem \ref{proj_geom}(3.2) $S_{\alpha} \cap S_{U}=[A^{\flat}]$. But then from $[A^{\flat}] \in S_{U}$, we conclude that $w(A^{\flat},-)$ is identically zero on $U$, which is a contradiction. Conversely assume $S_{U}$ is nonsingular, then by the above discussion there will be no roots of $det|_{\BP U}$ lying in $TS \backslash Sing(TS)$ and of order at least two. Thus to show $[U]$ is nondegenerate, it suffices to show that $\BP U \cap Sing(TS)= \varnothing$. Assume otherwise, then by Theorem 2.12(3.3) we have $dim(S_{x}) > 0$ for any $x \in Sing(TS) \cap \BP U$. Then as $S_{x} \cap S_{U}=S_{x} \cap H_{y}$ for any $y \in \BP U $ distinct from $x$, we conclude that it is non-empty and this yields a contradiction.
\end{proof}
The main application of Corollary 2.11 is the following:
\begin{cor}\label{section_}
	The  embedding $(Lie(H_{D_{4}}),W_{D_{4}}) \subset (Lie(H_{G}),W_{G})$ given in Theorem \ref{dict_02} induces a morphism between GIT quotients: $$\sigma_{G}: M=\CU_{D_{4}}//(H_{D_{4}} \times S_{3}) \rightarrow M_{G}=\CU_{G}//H_{G}$$.  
\end{cor}
\begin{proof}
By Lemma 2.14, the embedding induces an inclusion $\CU_{D_{4}} \subseteq \CU_{G}$. And from Corollary 2.11 there exists a group homomorphism:
$S_{3} \times H^{sc}_{D_{4}} \rightarrow \overline{H}_{\CH_{3}(\BA)}$ such that the inclusion is equivariant under the associated group actions. Then following Remark \ref{top_G} we conclude that it induces a morphism between GIT quotients.
\end{proof}
\subsection{Bhargava-Ho's correspondence}\label{cr_det}
In this subsection we first review the Bhargava-Ho's correspondence in Theorem \ref{cr}. Then we apply it to calculate the GIT stable locus of the action of $H_{D_{4}}=SL_{2} \times SL_{2} \times SL_{2}$ on $Gr(2,W_{D_{4}})=Gr(2,\BC^{2} \otimes \BC^{2} \otimes \BC^{2})$. We also fix a simple group $G$ of type $D_{4},F_{4},E_{6},E_{7}$ or $E_{8}$, and denote by $J,W,H$ the associated objects to $G$.\par 
\begin{defn}
	Given a nondegenerate plane $U \subseteq W$, there is a natural diagram of morphisms by Theorem \ref{proj_geom} (3.1)(3.2):\par 
\begin{equation}\label{diag1}
	\begin{tikzcd}
		C_{U} \arrow[r,"\pi^{*}\psi_{U}"] \arrow[d,"\pi_{U}"]
		& \mathcal{T} \arrow[d,"\pi"] \arrow[r,"v"] & S\\
		\mathbb{P}U \arrow[r,"\psi_{U}"]
		& Hilb_{2}(S)
	\end{tikzcd}    
\end{equation}
where $\pi$ is the universal family of subschemes of length two on $S$, and $v$ is the evaluation map to $S$. And the morphism $\psi_{U}$ is defined as follows.\par 
$\bullet$ For each point $x \in \mathbb{P}U \backslash TS$, $\psi_{U}(x)\dot =x_{p}+x_{q}$ where $l_{x_{p}x_{q}}$ is the unique secant line through $x$.\par 
$\bullet$ For each point $y \in \mathbb{P}U \cap TS$, denote by $y_{r}$ the point on $S$ such that $l_{y_{r},y}$ is the unique tangent line on $S$ through $y$. Then $\psi_{U}(y)$ is defined by the tangent direction $[T_{y_{r}}(l_{y_{r},y})]$. \par 
Finally $\pi_{U} \dot =\psi_{U}^{*} \pi $. Then the datum defined in Theorem \ref{cr} is given by the curve $C_{U}$, the morphism  $\kappa_{U}=v \circ \pi^{*} \psi_{U}$, and the line bundle $L_{U}=\pi_{U}^{*}\CO_{\BP U}(1)$.
\end{defn}
If $G$ is of type $D_{4}$ then $H=SL_{2} \times SL_{2} \times SL_{2}$, and $S=\BP^{1} \times \BP^{1} \times \BP^{1}$. The morphism $\kappa_{U}$ from $C_{U}$ to $S$ is given by three degree two maps to $\BP^{1}$, and hence given by three degree two line bundles on $C_{U}$. Now we can verify Corollary \ref{D_{4}}(1).
\begin{proof}[Proof of Proposition \ref{D_{4}}(1)]
	We first show that the action of $H_{D_{4}}$ on $\CU_{D_{4}}$ is GIT stable. By Theorem 1.7(1) it is already semi-stable, thus it suffices to show that the dimension of orbits in $\CU_{D_{4}}$  remain invariant. In fact we will show that $dim(Stab_{H_{D_{4}}}([U]))=0$ for any $[U] \in \CU_{D_{4}}$. Take any plane $U$, denote by $(C_{U},L_{U},L_{1},L_{2},L_{3})$ the associated datum. Then by Proposition \ref{D_{4}'} we obtain an injective group homomorphism:
	\begin{equation*}
		Stab_{H_{D_{4}}}([U]) \rightarrow Aut(C_{U},L_{U},L_{1},L_{2},L_{3}).
	\end{equation*}
As $Aut(C_{U},L_{U})$ is finite, it then suffices to check that the subgroup whose elements fix $(C_{U},L_{U})$ is finite. From the construction above this subgroup equals $\{h \in H_{D_{4}}=SL_{2} \times SL_{2} \times SL_{2}: h.\kappa_{U}=\kappa_{U}\}=\{\pm1\} \times \{\pm1\} \times \{\pm1\}$, concluding the proof of our claim. \par 
Conversely we show that any plane $[U]$ lying ouside $\CU_{D_{4}}$ is unstable. From the proof of Theorem \ref{3_lines}(1), there exists a point $x=[A] \in \BP U \cap TS_{S_{4}}$ such that $S_{x} \cap S_{U} \not= \varnothing$. Whence there is a point $z=[\alpha] \in S$, such that $w(\alpha,-)$ is identically zero on $U$ and $x \in \BP \widehat{T_{z}S}$. Now under the birational map (2.3) we can assume $z=Q_{J}([1:0])$ and thus $A=(a,b,0,0) \in \CC_{J}$. Take any point $y=[B] \in \BP U$ different from $x$, we can write $B=(a',b',c',0) \in \CC_{J}$ by formula (2.2). Then apply the $\BG_{m}$-action defined in (2.5), one checks that $\lim_{t \rightarrow 0}t[U] =\lim_{t \rightarrow 0}t[A \wedge B]$ exists hence $[U]$ is GIT unstable by the Hilbert-Mumford criterion \cite[Theorem 2.1]{MFK94}.
\end{proof}
\section{Three line bundles of degree two}
\subsection{From $(C,L,\kappa)$ to three line bundles of degree two on $C$} In this subsection we  are going to prove Theorem \ref{3_lines}.(2). 
Let $C$ be a smooth projective curve of genus one. To obtain a multiset of line bundles from a given morphism from $C$ to $S$, we first identify a set of degree two line bundles on $C$ by the effective divisors it defines on the smooth surface $Hilb_{2}(C)$ as follows.
	 Denote by $\Sigma: C \times C \rightarrow Hilb_{2}(C)$ the quotient map of the $S_{2}$-action on $C \times C$. Let $Pic(C)$ be the Picard scheme of $C$ and denote by $Jac^{2}(C)$ the component parametrizing line bundles of degree two on $C$. As $C$ is of genus one, $Jac^{2}(C)$ is also a smooth projective curve of genus one. Denote by $\CL^{[2]}$ the universal line bundle of degree two on the family $\pi_{2}:Jac^{2}(C) \times C \rightarrow Jac^{2}(C)$. Then $\CP^{[2]} := \BP((\pi_{2})_{*}(\CL^{[2]}))$ parametrizes the effective divisors of degree two on $C$, i.e., $\CP^{[2]} \cong Hilb_{2}(C)$. The $\BP^{1}$-bundle $\CP^{[2]}/Jac^{2}(C)$ induces the following natural morphism:
	 \begin{align*}
	 	\rho: Jac^{2}(C) &\rightarrow Div_{Hilb_{2}(C)},\\
	 	L &\rightarrow  |L|,
	 \end{align*}
 where for $L \in Jac^{2}(C)$, $|L| \cong \BP^{1}$ is its complete linear system,  and we identify an effective divisor in $|L|$ with a point in $Hilb_{2}(C)$.  Recall that for a variety $X$, the sum of Cartier effective divisors induces a natural addtion map on $Div_{X}$. Thus for any $m \geqslant 1$ we can define a morphism $\rho_{m}$ by:
 \begin{align*}
	\rho_{m}: Hilb_{m}(Jac^{2}(C) )&\rightarrow Div_{Hilb_{2}(C)},\\ \sum_{i=1}^{m} L_{i} &\rightarrow \sum_{i=1}^{m} |L_{i}|,
	\end{align*}
where we identify $Hilb_{m}(Jac^{2}(C))$ with the $m$-th symmetric product of $Jac^{2}(C)$ via the Hilbert-Chow morphism.
	\begin{lem}
		For any $m \geqslant 1$, $\rho_{m}$ is a closed embedding.
	\end{lem}
	\begin{proof}
		First by definition $\rho_{m}$ is injective. It is proper as $Hilb_{m}(Jac^{2}(C))$ is proper and $Div_{Hilb_{2}(C)}$ is separated.  It then suffices to show that every point in the image is nonsingular of dimension $m$. In fact for any $R \in Hilb_{m}(Jac^{2}(C) )$, we claim that $h^{0}(N_{\rho_{m}(R)/Hilb_{2}(C)})=m$, and $h^{1}(N_{\rho_{m}(R)/Hilb_{2}(C)})=0$. 
		 Given $R=\sum_{k=1}^{r}a_{k}L_{k}$ with $\sum_{k=1}^{r}a_{k}=m$, and $L_{i} \not= L_{j}$ for $i \not= j$. Then  $h^{i}(N_{\rho_{m}(R)/Hilb_{2}(C)})=\sum_{k=1}^{r} h^{i}(N_{\rho_{a_{k}}(a_{k}L_{k})/Hilb_{2}(C)})$. Thus it suffices to prove the claim for the case where $R=mL_{1}$ for some $L_{1} \in Jac^{2}(C)$. In this case $N_{\rho_{m}(R)/Hilb_{2}(C)} \cong \CO_{m|L_{1}|}(m|L_{1}|)$. We first prove by induction of $m \geqslant 1$ that $h^{1}(m|L_{1}|,\CO_{m|L_{1}|}(k|L_{1}|))=0$, for any integer $k$. If $m=1$, then $|L_{1}|$ is a smooth rational curve, and $C_{1}\doteq  \Sigma^{*}(|L_{1}|)$ is a smooth curve of genus one on $C \times C$, intersecting properly with the diagonal $\Delta_{C}$ at four points. Then $|L_{1}|\cdot |L_{1}|=C_{1}\cdot C_{1}/2=(K_{C \times C}+C_{1})\cdot C_{1}/2=deg(K_{C_{1}})/2=0$. Thus $h^{1}(|L_{1}|,\CO_{|L_{1}|}(k|L_{1}|))=h^{1}(|L_{1}|,\CO_{|L_{1}|})=0$, for any $k$. Now assume that it is true for some integer $m \geqslant 1$, then we consider the  following exact sequence:
		 \begin{equation*}
		 	0 \rightarrow \CO_{|L_{1}|}(-(m-1)|L_{1}|)  \rightarrow       \CO_{m|L_{1}|} \rightarrow \CO_{(m-1)|L_{1}|} \rightarrow 0.
		 \end{equation*}
	  Tensor it with  $\CO_{m|L_{1}|}(k|L_{1}|)$, then by assumption and  by the case $m=1$ we conclude the induction. Finally by the Riemann-Roch Theorem, for any $m \geqslant 1$:
	  \begin{align*} h^{0}(m|L_{1}|,\CO_{m|L_{1}|}(m|L_{1}|))&=\chi(m|L_{1}|,\CO_{m|L_{1}|}(m|L_{1}|))\\
	  	&=\chi(Hilb^{2}(C \times C),|mL_{1}|)-\chi(Hilb^{2}(C \times C),\CO_{Hilb^{2}(C \times C)})\\
	  	&=m|L_{1}|.(m|L_{1}|-K_{Hilb_{2}(C \times C)})/2\\
	  	&=-|mL_{1}|.K_{Hilb^{2}(C \times C)}/2=-mC_{1}.(K_{C \times C}-\Delta_{C})/4=m.
	  	\end{align*}
		\end{proof}
	\begin{defn}
		Let $C$ be a smooth projective curve of genus one.	For an effective divisor $E$ on $C \times C$, we say it comes from $m$ line bundles of degree two $R \in Hilb_{m}(Jac^{2}(C))$ if $E=\Sigma^{*}(\rho_{m}(R))$.
		\end{defn}
	Now we fix $G=D_{4},E_{6},E_{7},E_{8}$ or $F_{4}$. We denote by  $J$ the Jordan algebra associated to G, $S=S_{G}$ the subadjoint variety, and we identify $W=W_{G} \cong \CC_{J}$ as in Theorem 2.4. Given a nondegenerate plane $U \subseteq W=W_{G}$, we denote by
	$(C,L,\kappa: C \rightarrow S)$ the triple given in Theorem \ref{cr}. That is, in diagram  (\ref{diag1}), we let  $C=C_{U}, L=L_{U}$, and $\kappa=\kappa_{U}$. As $S \subseteq \mathbb{P}W$ is linearly normal, we view the skew bilinear form $w \in \wedge^{2}W^{\vee}$ as a section of $\CO_{S}(1) \boxtimes \CO_{S}(1)$.   We have the following diagram of morphisms:
	\begin{equation}\label{diag2}
		\begin{tikzcd}
			S \times S  & C \times C \arrow[r,"p_{j}"] \arrow[d,"\Sigma"] \arrow[l,"\kappa \times \kappa "]
			& C \arrow[d,"\pi_{U}"] \arrow[r,"\kappa"]  & S \\
			& Hilb_{2}(C) & \mathbb{P}^{1}
		\end{tikzcd}    
	\end{equation}
where $p_{j}$ is the projection to the $j$-th factor for $j=1,2$. For a point $P \in C$, we denote $C^{j}_{p}=p_{j}^{-1}(p)$. And we naturally identify it with $C$ through the maps  $\lambda_{j}: C \rightarrow C^{j}_{p}: Q \rightarrow (Q,P) \,\, \text{or}\,\, (P,Q)$.\par 
	The main result of this section is the following description of $(\kappa \times \kappa)^{*}(w)$, which implies Theorem \ref{3_lines}(2) combined with Theorem \ref{lines_S}(ii).
	\begin{prop}\label{3_line}
		$(\kappa \times \kappa)^{*}(w) \not=0$. Moreover $[(\kappa \times \kappa)^{*}(w)]=3\Delta_{C} + E_{1} + E_{2}+E_{3} $, where $\Delta_{C}$ is the diagonal, and: \par 
		(i) each $E_{i}$ is an irreducible smooth divisor intersecting with $\Delta_{C}$ properly at four points, and is mapped isomorphically onto $C$ by each $p_{j}$.\par 
		(ii) $E_{k}.E_{l}=0$, for any $1 \leqslant l,k \leqslant3$. And $E=E_{1}+E_{2}+E_{3}$ comes from three degree two line bundles $R_{E} \in Hilb_{3}(Jac^{2}(C))$.
	\end{prop}
We split the proof into several lemmas. The key lemma is as follows.
	\begin{lem}
		For any $P \in C$,  $(\kappa \times \kappa)^{*}(w)_{|_{C^{j}_{P}}} \not=0$, and $[(\kappa \times \kappa)^{*}(w)]_{|_{C^{j}_{P}}}=3P+\sum_{i=1}^{3}P_{i}$ for some points $P_{1},P_{2},P_{3}$ on $C$. Moreover  if $\pi_{U}$ is ramified at $P$, then 
		$P_{i} \not=P$, for any $i=1,2,3$.
	\end{lem}
	\begin{proof}
		First note that $(\kappa \times \kappa)^{*}(\CO_{S}(1) \boxtimes \CO_{S}(1))]_{|_{C^{j}_{P}}} \cong 3L$ is of degree 6. Thus it suffices to show that $$(\kappa \times \kappa)^{*}(w)_{|_{C^{j}_{P}}} \not=0\,\,\text{and}\,\,ord_{P}((\kappa \times \kappa)^{*}(w)_{|_{C^{j}_{P}}}) \geqslant 3,$$ 
		where the equality holds if $\pi_{U}$ is ramified at $P$.
		Denote $q=\kappa(P)$. Then by the homogeneity of $S$ we can assume $q=x_{0}=[A_{0}]$ in (\ref{bir_map}), where $A_{0}=(1,0,0,0) \in  \CC_{J}$. Denote by $[w_{x_{0}}]=[w(A_{0},-)] \in \mathbb{P}H^{0}(S,\CO_{S}(1))$ the class of effective Cartier divisor given by $x_{0}$.  Then one easily sees that $[(\kappa \times \kappa)^{*}(w)]_{|_{C^{j}_{P}}}=\kappa^{*}([w_{x_{0}}])$. Now if $P$ is an unramified point of $\pi_{U}$, denote $\overline{P}$ to be its conjugate. Then by the definition in (\ref{diag1}), $\pi_{U}(P)$ lies on the line $l_{\kappa(P),\kappa(\overline{P})}$. And as $U$ is nondegenerate, by Theorem \ref{lines_S}(iii), we conclude that $w(\kappa(P),\kappa(\overline{P}))\not=0$ hence $(\kappa \times \kappa)^{*}(w)_{|_{C^{j}_{P}}} \not=0$.
		\par 
		On the other hand, from the birational map (\ref{bir_map}), we can write $\kappa$ locally near $P$ as follows:
		\begin{align*}
			\kappa_{|\CV}:    \CV &\rightarrow \mathbb{P}(\mathbb{C} \oplus J \oplus J \oplus \mathbb{C}) \\
			v &\rightarrow [1: f(v):f(v)^{\#}:N(f(v))],
		\end{align*}
		where $\CV$ is a suitable affine open neighbourhood of $P$ in $C$, and $f$ is the unique morphism from $\CV$ to $J$ such that $\kappa_{|_{\CV}}$ factors through the natural embedding of $J$ inside $S$ as in (2.4). In particular $f(P)=0$. By the definition of $w$ in (\ref{w}), $\kappa_{|\CV}^{*}(w_{x_{0}})$ equals the divisor defined by the regular function $N \circ f$. As $N$ is a cubic polynomial over $J$, we conclude that $ord_{P}(\kappa_{|\CV}^{*}(w_{x_{0}})) \geqslant 3$. Now if $P$ is a ramified point of $\pi_{U}$, then by the definition of $\kappa$ in (\ref{diag1}), $\pi_{U}(P) \in TS \backslash Sing(TS)$ and $\pi_{U}(P)$ lies on the tangent line $l_{q,\pi_{U}(P)}$. Moreover the tangent direction $[Im(d\kappa_{|_{P}})]=[T_{q}l_{q,\pi_{U}(P)}]$ is generic by Theorem \ref{lines_S}(3), i.e., $N(f^{'}(q)) \not=0$. This implies that $\kappa_{|\CV}^{*}(w_{x_{0}}) \not=0$, and $ord_{P}(\kappa_{|\CV}^{*}(w_{x_{0}}))= 3$
	\end{proof}
	By the above lemma, $(\kappa \times \kappa)^{*}(w) \not=0$. We write $$D := [(\kappa \times \kappa)^{*}(w)]=a \Delta_{C} + \sum_{k=1}^{m}E_{k},$$
	where each $E_{i}$ is a prime divisor different from $\Delta$, and $E_{i}$ can be equal to $E_{j}$ for some $i,j$. And we write $E=\sum_{k=1}^{m}E_{k}$ as the non-diagonal part of $D$.
	\begin{lem}    $a=3,D^{2}=72,\Delta_{C}.E=12$, and $E^{2}=0$. Moreover $\chi(C \times C, \CO_{E})=0$
	\end{lem}
	\begin{proof}
		Denote by $p_{j}^{k}$ the restriction of $p_{j}$ to $E_{k}$ for each $j=1,2$ and for any $k$. Take any $P \in C \backslash (\bigcup\limits_{k} p_{j}^{k}(E_{i} \cap \Delta_{C}) )$, then $3 \leqslant mult_{P}(D\cap  C_{P}^{j})=a \Delta_{C}\cdot C_{P}^{j}=a$ by Lemma 3.4. On the other hand, take $P$ as a ramified point of $\pi_{U}$, then also by Lemma 3.4, $a=a \Delta_{C}\cdot C_{P}^{j} \leqslant ord_{P}(D_{|_{C_{P}^{j}}})=3$.  As $D \in |3L \boxtimes 3L|$,  $D^{2}=2 \cdot deg(3L)^{2}= 72$. $\Delta_{C}^{2}=(K_{C \times C}+\Delta_{C})\cdot \Delta_{C}=deg(K_{\Delta_{C}})=0$ as $C$ is of genus of one. Now $\Delta_{C}\cdot E=\Delta_{C}\cdot D=deg((3L \boxtimes 3L)_{|_{\Delta_{C}}})=deg(6L)=12$.  Then $E^{2}=D^{2}-2a\Delta_{C}\cdot E=0$. Finally by the Riemann-Roch Theorem we have $$\chi(C \times C, \CO_{E})=\chi(C \times C, \CO_{C \times C})-\chi(C \times C, \CO_{C \times C}(-E))=-(E+K_{C \times C})\cdot E/2=0$$.
	\end{proof}
	\begin{cor}
		We have $m \leqslant 3$. Moreover each $E_{k}$ is a smooth genus one curve on $C \times C$, and $p_{j}^{k}$ is a surjective \'etale morphism for any $j$. Finally $E_{k}\cdot E_{l}=0$, for any $1 \leqslant k,l \leqslant 3$.
	\end{cor}
	\begin{proof}
		For the first two parts, by Hurwitz's formula,  it suffices to show that each $E_{k}$ is smooth of genus one and $p_{j}^{k}$ is surjective.
		From Lemma 3.4, $D$ does not contain $C_{p}^{j}$ for any $P$ and $j$. Thus $p_{j}^{k}$ is surjective for each $k$. Denote by $d_{j}^{k}=deg(p_{j}^{k})$, then by Lemma 3.5 we have $\sum_{k=1}^{m}d_{j}^{k}=3$, and hence $m\leqslant 3$. If $d_{j}^{k}=1$ for each $k$, then we conclude that $m=3$, and each $p_{j}^{k}$ is an isomorphism. Now assume $m \leqslant 2$. Denote $q_{E}:\Tilde{E} \rightarrow E$ to be the normalization of $E$, and $Q_{E} \dot=(q_{E})_{*}(\CO_{\Tilde{E}})/\CO_{E}$ the quotient sheaf with finite support. Then $\chi(Q_{E})=\chi(\CO_{\Tilde{E}})$ as $\chi(C \times C, \CO_{E})=0$. If $m=1$ then $\Tilde{E}$ is irreducible, so $1-g(\Tilde{E})=\chi(\CO_{\Tilde{E}})=\chi(Q_{E}) \geqslant 0$. But $g(\Tilde{E}) \geqslant g(C) =1$. Thus we conclude that $g(\Tilde{E})=1$ hence $Q_{E}=0$, $q_{E}$ is an isomorphism, and $E$ is a smooth curve of  genus one. If $m=2$ then we can write $\Tilde{E}=\Tilde{E}_{1} \cup \Tilde{E}_{2} $, where each $\Tilde{E}_{k}$ is the normalization of $E_{k}$. And we can assume $d_{1}^{1}=1,d_{1}^{2}=2$. Thus $f_{1}^{1}$ is an isomorphism and $E_{1}$ is a smooth curve of genus one. Then $\chi(\CO_{\Tilde{E}})=\chi(\CO_{\Tilde{E_{1}}})+\chi(\CO_{\Tilde{E_{2}}})=1-g(\Tilde{E_{2}})=\chi(Q_{E}) \geqslant 0$. By the same discussion as for $m=1$ we conclude that $E_{2}$ is a smooth curve of genus one. For the last part, for each $E_{k}$ we have $E_{k}^{2}=(K_{C \times C}+E_{k})\cdot E_{k}=deg(K_{E_{k}})=0$. Thus $\sum_{E_{i}\not=E_{j}}E_{i}\cdot E_{j}=(\sum_{k}E_{k})^{2}=E^{2}=0$, concluding the proof.
	\end{proof}
	
	\begin{lem}
		Let $F$ be an irreducible smooth curve of genus one on $C \times C$, different from $\Delta_{C}$. If $F$ is $S_{2}$-invariant and if it is surjective \'etale over $C$ under each projection $p_{j}$, then it intersects properly with the diagonal $\Delta_{C}$.
	\end{lem}
	\begin{proof}
		Denote by $\sigma$ the generator of $S_{2}$.  Assume $b=(P,P) \in \Delta_{C} \cap F$, denote $\mathfrak{m}_{C,P} \,(resp. \, \mathfrak{m}_{F,b})$ to be the maximal ideal of local ring $\CO_{C,P}\,(resp.\,  \CO_{F,b})$, and let $t_{b}\, (resp.\,  f_{b})$ be a parameter of the maximal ideal. 
		Note that the ideal $(I_{\Delta_{C}})_{|_{b}} \subseteq \CO_{C \times C, b}$ is generated by $t_{b}\otimes 1-1 \otimes t_{b}$. As $F$ is \'etale over $C$ and $S_{2}$-invariant, $(I_{\Delta_{C}})_{|_{F,b}}$ is then generated by $f_{b}-\sigma(f_{b})$. As $\sigma$ fixes $b$, $\sigma$ acts on $\CO_{F,b}$ and  we can assume $\sigma(f_{b})=cf_{b}+u f_{b}^{2}$ for some $c \in \mathbb{C},u \in \CO_{F,b}$. Then applying $\sigma^{2}=1$ we get $c=\pm 1$. Thus it suffices to show that $c=-1$.  If otherwise, then we claim that $\sigma$ fixes $\CO_{F,b}$, whence $F$ is fixed and thus $F=\Delta_{C}$, which is a contradiction. To show the claim, we first prove by induction that $\sigma$ fixes $\CO_{F,b}/\mathfrak{m}_{F,b}^{k}$ for each positive integer $k$, then the claim follows from the fact that $\CO_{F,b}$ is $\sigma$-equivariantly included in the complete local ring $\widehat{\CO}_{F,b}$. Now assume by induction that  $\sigma$ fixes  $\CO_{F,b}/\mathfrak{m}_{F,b}^{k}$ for some $k \geqslant 2$, then consider the $\sigma$-equivariant exact sequence:
		\begin{equation*}
			0 \rightarrow \mathfrak{m}_{F,b}^{k}/\mathfrak{m}_{F,b}^{k+1} \rightarrow  \CO_{F,b}/\mathfrak{m}_{F,b}^{k+1} \rightarrow \CO_{F,b}/\mathfrak{m}_{F,b}^{k} \rightarrow 0.
		\end{equation*}
		By assumption for every $\alpha \in \CO_{F,b}/\mathfrak{m}_{F,b}^{k+1}$, $\sigma(\alpha) -\alpha \in \mathfrak{m}_{F,b}^{k}/\mathfrak{m}_{F,b}^{k+1}$. But $\sigma(\sigma (\alpha) -\alpha)=-(\sigma(\alpha) -\alpha)$ and $\sigma$ fixes $\mathfrak{m}_{F,b}^{k}/\mathfrak{m}_{F,b}^{k+1} \cong Sym^{k}(\mathfrak{m}_{F,b}/\mathfrak{m}_{F,b}^{2})$. Thus we conclude that $\sigma(\alpha)=\alpha$, the claim is proved.
	\end{proof}
	Recall that $\Sigma$ is a ramified double cover  with ramified divisor $\Delta_{C}$. Then any $S_{2}$-invariant effective divisor on $C \times C$ whose components do not include $\Delta_{C}$ must be a pull-back of an effective divisor on $Hilb_{2}(C)$. Denote by $F$ the effective divisor on $Hilb_{2}(C)$ such that $\Sigma^{*}(F)=E$.
	\begin{lem}
		$F^{2}=0$, and $\chi(Hilb_{2}(C),\CO_{F})=3$.
	\end{lem}
	\begin{proof}
		$F^{2}= \Sigma^{*}(F)\cdot \Sigma^{*}(F)/2= E^{2}/2=0$ by Lemma 3.5. And $F\cdot K_{Hilb_{2}(C)}=E\cdot (K_{C \times C}-\Delta_{C})/2=-E\cdot\Delta_{C}/2=-6$. Thus by the Riemann-Roch Theorem: 
		\begin{equation*}
			\chi(Hilb_{2}(C),\CO_{F})=\chi(\CO_{Hilb_{2}(C)})-\chi(\CO(-F))=-F\cdot (F+K_{Hilb_{2}(C)})/2=3.        
		\end{equation*}
		
	\end{proof}
	Now we are ready to prove Proposition 3.3.
	\begin{proof}[Proof of Proposition 3.3]
		First we show that $m=3$. Assume that $m=1$, then $F$ is a prime divisor. So $\chi(\CO_{F})=1-h^{1}(\CO_{F}) \leqslant 1$, contradicting Lemma 3.8. If $m=2$ then $E=E_{1}+E_{2}$ with $E_{1} \not =E_{2}$, as $d_{j}^{1} \not=d_{j}^{2}$. Then both of them are $S_{2}$-invariant. Otherwise from $E_{1}\cdot E_{2}=0$  we have $E_{1}\cdot\Delta_{C}=E_{2}\cdot\Delta_{C}=0$, contradicting  $E\cdot \Delta_{C}=12$. Thus for each $k$ we can assume $E_{k}=\Sigma^{*}(F_{k})$ for some prime divisor $F_{k}$ on $Hilb_{2}(C)$. And we have $F_{1}\cdot F_{2}=E_{1}\cdot E_{2}/2=0$. But then $\chi(\CO_{F})=\chi(\CO_{F_{1}})+\chi(\CO_{F_{2}}) \leqslant 2$, again a contradiction. Thus $m=3$. And from the proof of Corollary 3.6, we conclude that each $p_{j}^{k}$ is an isomorphism.\par 
		Then we  show that each $E_{k}$ intersects properly $\Delta_{C}$ at four points applying Lemma 3.7. We first show that each $E_{k}$ is $\sigma$-invariant. Assume otherwise, then we can assume  $\sigma(E_{1})=E_{2}$ with $E_{1} \not= E_{2}$, and hence $E_{3}$ is $\sigma$-invariant as $E$ is $\sigma$-invariant. Then there is a prime divisor $F'$(resp. $F_{3}$) on $Hilb_{2}(C)$ such that $\Sigma^{*}(F')=E_{1}+E_{2}$(resp. $\Sigma^{*}(F_{3})=E_{3}$), with $F'+F_{3}=F$. Moreover $F' \cdot F_{3}=(E_{1}+E_{2})\cdot E_{3}/2=0$. Thus $\chi(\CO_{F})=\chi(\CO_{F_{3}})+\chi(\CO_{F'}) \leqslant 2$, a contradiction. Now we can assume for each $k$ that there is a prime divisor $F_{k}$ such that $\Sigma^{*}(F_{k})=E_{k}$, and $F=F_{1}+F_{2}+F_{3}$. Then by similar calculations as in Lemma 3.8, we conclude that $E_{k}\cdot \Delta_{C}/4=\chi(\CO_{F_{k}}) \leqslant 1$ for each $k$.  Thus from $\Delta_{C}\cdot E=12$ we get $\Delta_{C}\cdot E_{k}=4$, and $\chi(\CO_{F_{k}})=1$ for each $k$.  In particular this implies that $F_{k} \cong \mathbb{P}^{1}$ for each $k$. Then $\kappa_{k,j}:(p_{j}^{k})^{-1}\circ  S_{|_{E_{k}}}: C \rightarrow F_{k}$ is a ramified double cover of $\mathbb{P}^{1}$ for each $k$.  Denote by $L_{k}=\kappa_{k,j}^{*}(\CO(1))$. Then one easily checks that $F_{k}=\rho_{1}(L_{k})$ for each $k$, and thus $F=\rho_{3}(L_{1}+L_{2}+L_{3})$. The proof is complete.
	\end{proof}
\subsection{Relative setting}
    Throughout this subsection $G$ is supposed to be of type $F_{4},E_{6},E_{7}$ or $E_{8}$. Let $\CU_{G} \subseteq Gr(2,W_{G})$ be the open subset of nondegenerate planes. Recall from (\ref{diag1}) that, every nondegenerate plane $[U] \in \CU_{G}$ is associated with a datum $(C_{U},L_{U},\kappa_{U}: C_{U} \rightarrow S_{G})$. As $U$ varies in $\CU_{G}$, it then defines a family of data over $\CU_{G}$. We formulate it as follows.\par 
    Recall that $\pi: \CT \rightarrow Hilb_{2}(S)$ is the universal family of subschemes of length two on $S$. And $v$ is the evaluation map to $S$. From Theorem \ref{proj_geom} (3),  one can easily see that there is an $H_{G}$-equivariant morphism $\gamma_{G}:\mathbb{P}W_{G} \backslash Sing(TS_{G}) \rightarrow Hilb_{2}(S_{G})$, such that for every nondegenerate plane $U$, the morphism  $\psi_{U}$ defined in (\ref{diag1}) coincides with the composition of $\gamma_{G}$ and the inclusion of $\mathbb{P}U$ in $\mathbb{P}W_{G}$. 
Denote $\eta_{G}: \CP_{G} \rightarrow \CU_{G}$ to be the universal bundle of lines in $\mathbb{P} W_{G}$. Let $\CO_{\CP_{G}}(1)$ be the tautological line bundle of $\CP_{G}$. And let $\mu_{G}: \CP_{G} \rightarrow \mathbb{P}W_{G}$ be the evaluation map.  Then  by our definition of nondegenerate planes, $Im(\mu_{G})$ is included in $ \mathbb{P}W_{G}\backslash Sing(TS_{G})$. Thus one can define a similar diagram of morphisms as (\ref{diag1}):
\begin{equation}\label{Diag1}
	\begin{tikzcd}
		\CC_{G} \arrow[r,"\pi^{*}\psi_{G}"] \arrow[d,"\pi_{G}"]
		& \mathcal{T} \arrow[d,"\pi"] \arrow[r,"v"] & S\\
		\CP_{G} \arrow[r,"\psi_{G}"]
		& Hilb_{2}(S),
	\end{tikzcd}    
\end{equation}
where $\psi_{G}=\gamma_{G} \circ \mu_{G}$, and $\pi_{G}=\psi_{G}^{*}(\pi)$. We denote $\CL_{G}=\pi_{G}^{*}(\CO_{\CP_{G}}(1))$,  $\kappa_{G}=v \circ \pi^{*}\psi_{G}$, and  $K_{G}=\eta_{G} \circ \pi_{G}$. Then one sees that $K_{G}$ is a smooth proper family of curves of genus one, $\CL_{G}$ is a family of degree two line bundles on $\CC_{G}/\CU_{G}$. 
The datum $(K_{G}: \CC_{G } \rightarrow \CU_{G}, \CL_{G},\kappa_{G})$ defines our desired family of nondegenerate morphisms from polarized curves of genus one to $S_{G}$. \par 
Next we aim to formulate our results in Section 3.1 under this relative setting. First we fix some notations. For any $[U] \in \CU_{G}$, we denote by  $E_{U}$ the non-diagonal part of $(\kappa_{U} \times \kappa_{U})^{*}(w)$, $F_{U}$ the effective divisor on $Hilb_{2}(C_{U})$ such that $\Sigma_{U}^{*}(F_{U})=E_{U}$, $\rho_{m,U}=\rho_{m}$ the morphism defined before for $C=C_{U}$, and $R_{U} \in Hilb_{3}(Jac^{2}(C_{U}))$ such that $\rho_{3,U}(R_{U})=F_{U}$.\par 
  As $\CC_{G}/\CU_{G}$ is a smooth projective family of curves of genus one, the functor $Pic_{(\CC_{G}/\CU_{G})(\text{\'{e}tale})}$ is represented by a scheme $Pic_{\CC_{G}/\CU_{G}}$\cite[Theorem 9.4.8(1)]{FGA_explained}. Moreover for any $k \in \BZ$, denote $Pic_{(\CC_{G}/\CU_{G})(\text{\'{e}tale})}^{k}$ to be the $\text{\'{e}tale}$ subsheaf associated to the presheaf whose $T$-point for a scheme $T$ over $\CU_{G}$ is represented by line bundles on $\CC_{G} \times_{\CU_{G}} T/ T$ of degree $k$. Then from the proof of \cite[Theorem 9.4.8(2)]{FGA_explained}, $Pic_{(\CC_{G}/\CU_{G})(\text{\'{e}tale})}^{k}$ is an open subsheaf represented by an open subscheme $Pic_{\CC_{G}/\CU_{G}}^{k} \subseteq Pic_{\CC_{G}/\CU_{G}}$ which is quasi-projective over $\CU_{G}$, and $Pic_{\CC_{G}/\CU_{G}}$ is the disjoint union of $Pic_{\CC_{G}/\CU_{G}}^{k}$ for $k \in \BZ$.
\begin{lem}
	For any $k \in \BZ$, $Pic_{\CC_{G}/\CU_{G}}^{k}$ is a smooth projective  family of curves of genus one  over $\CU_{G}$.
\end{lem}
\begin{proof}
	The fact that $Pic_{\CC_{G}/\CU_{G}}$ is smooth over $\CU_{G}$ follows from \cite[Proposition 9.5.19]{FGA_explained} and that $\CC_{G}/\CU_{G}$ is a smooth family of curves. It is projective over $\CU_{G}$ by a general fact that for a smooth projective family of irreducible varieties $X/S$, any closed subscheme of $Pic_{X/S}$ is projective over $S$ if it is of finite type(cf. \cite[Exercise 9.5.7]{FGA_explained}). Finally for any point $[U] \in \CU_{G}$,  $Pic_{\CC_{G}/\CU_{G}}^{k}|_{[U]} \cong Jac^{k}(C_{U})$ for each $k \in \BZ$, which is a smooth projective curve of genus one. Thus it follows that $Pic_{\CC_{G}/\CU_{G}}^{k}$ is a smooth projective family of curves of genus one.
\end{proof}
Denote by $S^{m}_{\CU_{G}}(Pic^{2}_{\CC_{G}/\CU_{G}})$ the family of $m$  line bundles of degree two on $\CC_{G}/\CU_{G}$, which is defined as the quotient of the permutation action of $S_{m}$ on the m-tuple relative product of $Pic^{2}_{\CC_{G}/\CU_{G}}/\CU_{G}$. Then one easily sees that it is also smooth projective over $\CU_{G}$. Over any point $[U] \in \CU_{G}$, its fiber is isomorphic to $Hilb_{m}(Jac^{2}(C_{U}))$ via the Hilbert-Chow morphism. We aim to prove the following:
\begin{prop}\label{family}
	There is a section $\CR_{G}: \CU_{G} \rightarrow S^{3}_{\CU_{G}}(Pic^{2}_{\CC_{G}/\CU_{G}})$, such that for any $[U] \in \CU_{G}$, $\CR_{G}([U])=R_{U}$.
\end{prop}
Consider the action of $S_{2}$ on $\CC_{G} \times_{\CU_{G}} \CC_{G}$ given by permutations of factors. We define $S_{G}: \CC_{G} \times \CC_{G} \rightarrow S^{2}_{\CU_{G}}(\CC_{G})\dot=(\CC_{G} \times_{\CU_{G}} \CC_{G})//S_{2}$ to be its GIT quotient map. Then  we can get a similar diagram as  (\ref{diag2}):
	\begin{equation}\label{Diag2}
	\begin{tikzcd}
		S \times S  & \CC_{G} \times \CC_{G} \arrow[r,"p_{j,G}"] \arrow[d,"\Sigma_{G}"] \arrow[l,"\kappa_{G} \times \kappa_{G} "]
		& \CC_{G} \arrow[d,"K_{G}"] \arrow[r,"\kappa_{G}"]  & S \\
		& S^{2}_{\CU_{G}}(\CC_{G}) \arrow[r, "q_{G}"]& \CU_{G}
	\end{tikzcd}    
\end{equation}
By Propostion 3.3, $(\kappa_{G} \times \kappa_{G})^{*}(\omega) \not=0$. We write $[(\kappa_{G} \times \kappa_{G})^{*}(\omega)]=a\Delta_{\CC_{G}}+\CE_{G}$, where $\Delta_{\CC_{G}}$ is the diagonal, and $\CE_{G}$ is an effective divisor whose component does not contain $\Delta_{\CC_{G}}$.  Recall that for a scheme $X/S$, a relative effective divisor on $X/S$ is an effective divisor on $X$ which is flat over $S$.
\begin{cor}
We have	a=3. Moreover $\CE_{G}$ is a relative effective divisor on $(\CC_{G} \times \CC_{G})/\CU_{G}$, and there is a relative effective divisor $\CF_{G}$ on $S^{2}_{\CU_{G}}(\CC_{G}) / \CU_{G}$ such that $\CE_{G}=\Sigma_{G}^{*}(\CF_{G})$.  For any $[U] \in \CU_{G}$, we have $\CE_{G}|_{[U]}=E_{U}$ and $\CF_{G}|_{[U]}=F_{U}$. 
	\end{cor}
\begin{proof}
	As $\Delta_{\CC_{G}}$ is not contained in the component of $\CE_{G}$, there exists a nondgenerate plane $[U] \in \CC_{G}$ such that $\CE_{G}|_{[U]}$ does not contain $\Delta_{C_{U}}$.  Then it follows from Proposition 3.3 that $a=3$. And for any $[U] \in \CU_{G}$, one has $\CE_{G}|_{[U]}=E_{U}$, which is an effective Cartier divisor on $C_{U} \times C_{U}$. Thus $\CE_{G}$ is a relative effective divisor on $(\CC_{G} \times \CC_{G})/\CU_{G}$ \cite[Lemma 9.3.4(iii)]{FGA_explained}. Moreover as $\CE_{G}$ is obviously $S_{2}$-invairant, there is is an  effective divisor $\CF_{G}$ on $S^{2}_{\CU_{G}}(\CC_{G})$ such that $\CE_{G}=S_{G}^{*}(\CF_{G})$. Then for each $[U] \in \CU_{G}$, $\CF_{G}|_{[U]}=F_{U}$. And thus  $\CF_{G}$ is a relatively effective divisor on $S^{2}_{\CU_{G}}(\CC_{G}) / \CU_{G}$.
\end{proof}
   
   Note that there is a natural morphism $\Theta_{G,m}: S^{m}_{\CU_{G}}(Pic^{2}_{\CC_{G}/\CU_{G}}) \rightarrow Div_{S^{2}_{\CU_{G}}(\CC_{G})/\CU_{G}}$ for each $m \geqslant 1$, such that for any $[U] \in \CU_{G}$ we have $\Theta_{G,m}|_{[U]}=\rho_{m,U}$. Applying Lemma 3.1 we get:
   \begin{cor}
   	For each $m \geqslant 1$, $\Theta_{G,m}$ is a closed embedding.
   	\end{cor}
   \begin{proof}
   	By Lemma 3.1, $\Theta_{G,m}|_{[U]}$ is a closed embedding for any $[U] \in \CU_{G}$. Then it suffices to show that $\Theta_{G,m}$ is proper.
   As $S^{2}_{\CU_{G}}(\CC_{G})$ is projective and smooth over $\CU_{G}$,	$Div_{S^{2}_{\CU_{G}}(\CC_{G})/\CU_{G}}$ is a locally quasi-projective scheme over $\CU_{G}$\cite[Theorem 9.3.7]{FGA_explained}. In particular it is separated over $\CU_{G}$. Thus $\Theta_{G,m}$ is proper as $S^{m}_{\CU_{G}}(Pic^{2}_{\CC_{G}/\CU_{G}})$ is proper over $\CU_{G}$.  
   	\end{proof}
   \begin{proof}[Proof of Proposition 3.10]
  The relative effective divisor $\CF_{G}$  on $S^{2}_{\CU_{G}}(\CC_{G})/\CU_{G}$ defines a section $\tau_{G}: \CU_{G} \rightarrow Div_{S^{2}_{\CU_{G}}(\CC_{G})/\CU_{G}}$. From Corollary 3.11, we have $Im(\tau_{G}) \subseteq Im(\Theta_{G,3})$. Then the proof follows from the fact that $\Theta_{G,m}$ is a closed embedding and both $S^{m}_{\CU_{G}}(Pic^{2}_{\CC_{G}/\CU_{G}})$ and $\CU_{G}$ are varieties.
   	\end{proof}
    \section{proof of the main theorem}
    In this section we are going to prove our Main Theorem. We first give the precise definition of the moduli functor $\CM$ as follows, where we denote by $(Sch/\mathbb{C})_{fppf}$ the site consisting of the category of locally noetherian schemes over $\mathbb{C}$ equipped with $fppf$ topology.

    
    \begin{defn}\label{mod_func}
    	Denote by $\mathcal{M}$ the sheaf associated to the following presheaf:
    	\begin{align}
    		P\mathcal{M}: (Sch/\mathbb{C})^{op}_{fppf} &\rightarrow (Sets) \nonumber \\
    		T &\rightarrow \{K: \mathcal{C} \rightarrow T;\CL, \CR\}/\sim.
    		\nonumber
    	\end{align}
    	Here $K$ is a smooth proper family of curves of genus one. $\CL \in Pic(\CC/T)(T)$ is a family of line bundles of degree two on $\CC/T$, and $\CR: T \rightarrow S^{3}_{T}(Pic^{2}(\CC/T))$ is a section of the third symmetric product of Picard scheme of degree two of the family C/T. Moreover for any point $t \in  T$, denote by $\CR_{t}=(L_{1,t},L_{2,t},L_{3,t})$ the corresponding three line bundles on $\CC_{t}$. Then we also require that $\CL_{t} \cong L_{i,t}$ for any $i=1,2,3$, and that $L_{1,t}+L_{2,t}+L_{3,t}=3\CL_{t}$.
    	 Two data $(\CC/T,\CL,\CR)$ and $(\CC'/T,\CL',\CR')$ are said to be equivalent if and only if there is a $T$-isomorphism $\Psi: \CC \rightarrow \CC'$, such that $\Psi^{*}(\CL') \cong \CL$, and $\Psi^{*}(\CR')=\CR$
    \end{defn}
    We can verify Proposition \ref{D_{4}}(2).
    \begin{proof}[Proof of Proposition \ref{D_{4}}(2)]
    	The proof is adapted from the proof of \cite[Theorem 6.28]{BH}. Given a scheme $T$ and given an equivalent class in $P\CM(T)$ represented by $(K: \CC/T,\CL,\CR) $.  Denote by
    	\begin{equation*}
    		\CQ:Pic^{2}_{\CC/T} \times_{T} Pic^{2}_{\CC/T} \times_{T} Pic^{2}_{\CC/T} \rightarrow  S^{3}_{T}(Pic^{2}_{\CC/T})
    	\end{equation*}
    	the quotient map of the $S_{3}$-action. Denote by 
    	\begin{equation*}
    		\CR^{*}(\CQ): \widetilde{T} \rightarrow T, \,\,\text{and}\,\,\, \CQ^{*}(\CR): \widetilde{T} \rightarrow Pic^{2}_{\CC/T} \times_{T} Pic^{2}_{\CC/T} \times_{T} Pic^{2}_{\CC/T} 
    	\end{equation*}
    	the pull-backs.
    	Then $\widetilde{T}$ inherits an $S_{3}$-action, and $\CR^{*}(\CQ)$ is its quotient map. Moreover $\CQ^{*}(\CR)$ naturally defines three families of line bundles $\CL_{1},\CL_{2},\CL_{3}$ on the family $\widetilde{K}:\CC_{\widetilde{T}}=\CC \times_{T} \widetilde{T} /\widetilde{T}$ up to an $\text{\'{e}tale}$ cover of $\widetilde{T}$. Then from $fpqc$ descent of morphisms we can assume that the three line bundles are defined over $\widetilde{T}$.  Denote by $\widetilde{\CL}$ the pullback of $\CL$ to the family $\CC_{\widetilde{T}}/\widetilde{T}$.\par 
     Now denote by $\CP=\mathbb{P}_{\widetilde{T}}(\widetilde{K}_{*}(\widetilde{\CL}))$, let $\pi_{\widetilde{T}}: \CC_{\widetilde{T}} \rightarrow \CP$ be the double cover, $\CP_{i}=\mathbb{P}_{\widetilde{T}}(\widetilde{K}_{*}(\CL_{i}))$ for $i=1,2,3$, and $\CO_{\CP}(1)$ the tautological line bundle on $\CP$. Choose a Zariski covering $u: \overline{T} \rightarrow \widetilde{T}$ such that there exists an isomorphism $\widetilde{K}_{*}(\CL_{i})|_{\overline{T}} \cong \CO_{\overline{T}}^{\oplus 2}$ for each $i$.  It then induces a $\overline{T}$-isomorphism $\tau_{i}: \CP_{i} \times_{T} \overline{T} \cong \overline{T} \times \mathbb{P}^{1}$ for each $i$. Thus it gives a morphism $\kappa_{\overline{T}}: \CC_{\overline{T}} = \CC \times_{T} \overline{T} \rightarrow \mathbb{P}^{1} \times \mathbb{P}^{1} \times \mathbb{P}^{1}$, such that $\CL_{i}|_{\overline{T}} \cong (p_{i} \circ \kappa_{\overline{T}})^{*}(\CO_{\mathbb{P}^{1}}(1))$, where $p_{i}$ is the projection of $\mathbb{P}^{1} \times \mathbb{P}^{1} \times \mathbb{P}^{1}$ to its $i$-th factor. By our assumption and the construction (\ref{diag1}), there exists a morphism $ \psi_{\overline{T}}: \CP_{\overline{T}}=\overline{T} \times_{T} \CP \rightarrow Hilb_{2}(\mathbb{P}^{1} \times \mathbb{P}^{1} \times \mathbb{P}^{1})$ such that the ramified double cover $\pi_{\overline{T}}$ is the pull back of universal subscheme of lenghth two. Denote by $Li$ the universal scheme of lines connecting two points on $\mathbb{P}^{1} \times \mathbb{P}^{1} \times \mathbb{P}^{1}$. Let $e_{\overline{T}}: Li_{\overline{T}}=\psi_{\overline{T}}^{*}(Li) \rightarrow \CP_{\overline{T}}$ be its pull back, and let $\CO_{Li_{\overline{T}}}(1)$ be the tautological line bundle on $Li_{\overline{T}}$.  We describe the above morphisms in the following diagram:
    	\begin{equation}\label{Diag1}
    		\begin{tikzcd}
    			\CC_{\overline{T}} \arrow[r,"\pi^{*}\psi_{\overline{T}}"] \arrow[d,"\pi_{\overline{T}}"]
    			& \mathcal{T} \arrow[d,"\pi"]\arrow[r,"v"] & \mathbb{P}^{1} \times \mathbb{P}^{1} \times \mathbb{P}^{1}\\
    			\CP_{\overline{T}} \arrow[r,"\psi_{\overline{T}}"]
    			& Hilb^{2}(\mathbb{P}^{1} \times \mathbb{P}^{1} \times \mathbb{P}^{1}) \\
    			Li_{\overline{T}} \arrow[u,"e_{\overline{T}}"] \arrow[r]  & Li \arrow[u,"e"] \arrow[r] & \mathbb{P}(\mathbb{C}^{2} \otimes \mathbb{C}^{2} \otimes \mathbb{C}^{2})
    		\end{tikzcd}    
    	\end{equation}
    	
    	Then by our assumption, $Li_{\overline{T}}=\mathbb{P}_{\CP_{\overline{T}}}((\pi_{\overline{T}})_{*}\kappa_{\overline{T}}^{*}(\CO_{\mathbb{P}^{1}}(1) \boxtimes \CO_{\mathbb{P}^{1}}(1)  \boxtimes \CO_{\mathbb{P}^{1}}(1) ))=\mathbb{P}_{\CP_{\overline{T}}}((\pi_{\overline{T}})_{*}(\CO_{\CC_{\overline{T}}}) \otimes \CO_{\CP_{\overline{T}}}(3))$. As $\pi_{\overline{T}}$ is a ramified double cover, $(\pi_{\overline{T}})_{*}(\CO_{\CC_{\overline{T}}})=\CO_{\CP_{\overline{T}}} \oplus \CA_{-2}$ for  a line bundle $\CA_{-2}$ of degree $(-2)$ on the $\mathbb{P}^{1}$-bundle $\CP_{\overline{T}}/\overline{T}$. Then the projection $(\pi_{\overline{T}})_{*}(\CO_{\CC_{\overline{T}}}) \otimes \CO_{\CP_{\overline{T}}}(3) \rightarrow \CA_{-2} \otimes \CO_{\CP_{\overline{T}}}(3)$ yields a section $s_{\overline{T}}: \CP_{\overline{T}} \rightarrow Li_{\overline{T}}$, and hence we obtain a map  $l_{\overline{T}}:\CP_{\overline{T}} \rightarrow \mathbb{P}(\mathbb{C}^{2} \otimes \mathbb{C}^{2} \otimes \mathbb{C}^{2})$. Restricted to any $v \in \overline{T}$, one sees that the above construction is the same as \cite[Construction 6.30]{BH}. In particular $(\l_{\overline{T}})|_{v}$ is a linear embedding for any $v \in \overline{T}$. Thus it induces a morphism $\beta_{\overline{T}}: \overline{T} \rightarrow \CU_{D_{4}} \subseteq  Gr(2,\mathbb{C}^{2} \otimes \mathbb{C}^{2} \otimes \mathbb{C}^{2})$, such that for any $v \in \overline{T}$, the image of $v$ is exactly the nondegenerate plane given by the datum $(\kappa_{\overline{T}}|_{v},\CL|_{v})$ as in Theorem \ref{cr}. Then we get an induced morphism $\beta'_{\overline{T}}:\overline{T} \rightarrow M=\CU_{D_{4}}//H_{D_{4}} \times S_{3}$.  We now check that this morphism can be uniquely descended to a morphism from $T$ to $M$. First if we choose any other Zariski cover $\overline{T}_{2} \rightarrow \widetilde{T}$, we note that the pullbacks of $\beta_{\overline{T}}'$ and $\beta_{\overline{T}_{2}}'$ to the fiber product $\overline{T} \times_{\widetilde{T}} \overline{T}_{2}$ differ from a family of translations in $PGL_{2} \times PGL_{2} \times PGL_{2}$, i.e., induced by a morphism $\overline{T} \times_{\widetilde{T}} \overline{T}_{2} \rightarrow PGL_{2} \times PGL_{2} \times PGL_{2}$. Thus by the universal property of the GIT quotient by $H_{D_{4}}$ and by the $fpqc$ descent of morphisms, we obtain a unique morphism: $\beta_{\widetilde{T}}':\widetilde{T} \rightarrow M$.  Next for the finite cover $\widetilde{T} \rightarrow T$, this is the quotient morphism of the $S_{3}$-action on $\widetilde{T}$. From our construction one sees that $\beta_{\widetilde{T}}'$ is $S_{3}$-invariant, whence we obatin a unique morphism $\beta'_{T}: T \rightarrow M$ descending $\beta'_{\widetilde{T}}$.  	
    	One can then easily check that the above construction does not depend on the choice of the representing datum. It
    	 induces a natural transformation $\xi: \CM \rightarrow h_{M}$ which satisfies the universal property of coarse moduli spaces by considering the natural family $(\CC_{D_{4}}/\CU_{D_{4}},\CL_{D_{4}},\CR_{D_{4}})$. Finally the fact that $\xi(Spec(\BC))$ is bijective follows from Proposition \ref{D_{4}'}(2).
    	\end{proof}
    \begin{proof}[Proof of the Main Theorem]
    Now given $G=F_{4},E_{6},E_{7}$ or $E_{8}$, we have constructed in Section 3 a datum associated to $\CU_{G}$, namely $\alpha_{G} = (\CC_{G}/\CU_{G},\CL_{G},\CR_{G})$. Thus by Proposition \ref{D_{4}}(2) we obtain a morphism $\xi(\CU_{G})(\alpha_{G}): \CU_{G} \rightarrow M$,  such that for any nondegenerate plane $[U] \in \CU_{G}$, its image is the equivalent class of nondegenerate planes in $\mathbb{C}^{2} \otimes \mathbb{C}^{2} \otimes \mathbb{C}^{2}$ given by the datum $(C_{U},\pi_{U},L_{U})$ as in Theorem \ref{3_lines}(2).  The morphism is $H_{G}$-invariant as $H_{G}$-equivalent planes have isomorphic data. We then obtain a morphism between GIT quotients of nonsingular sections $\Phi_{G}:M_{G} \rightarrow M$. On the other hand recall from Section 2.2 that there is a natural embedding of representations: $(Lie(H_{D_{4}}),W_{D_{4}}) \subset (Lie(H_{G}),W_{G})$. For a plane $[U] \in \CU_{D_{4}}$, one sees that the associated datum $(C_{U},L_{U},L_{1},L_{2},L_{3})$ given by Proposition \ref{D_{4}'} is the same as the datum given by Theorem \ref{3_lines}(2) if we view $[U]$ as a nondegenerate plane in $W_{G}$ as well. In other words we can conclude that the natural morphism $\sigma_{G}: M \rightarrow M_{G}$ constructed in Corollary \ref{section_} is a section of $\Phi_{G}$. Finally by Theorem \ref{mot_}, both $M$ and $M_{G}$ are varieties of dimension 3. It follows that the section $\sigma_{G}$ is an isomorphism.
 \end{proof}
\appendix
\section{The open embedding of $M \subseteq \BP(1,3,4,6)$}
In this appendix we follow \cite[Section 2.4]{BH} to realize $M=M_{D_{4}}//S_{3}$ as an open subset of $\BP(1,3,4,6)$ whose complement is a hypersurface of degree 12.  The main ingredient is to consider the following moduli functor as indicated by Proposition 1.6(3):
\begin{defn}\label{moduli_func}
Denote by $\CM'$ the sheaf associated to the following presheaf:
\begin{align}
	P\mathcal{M}': (Sch/\mathbb{C})^{op}_{fppf} &\rightarrow (Sets) \nonumber \\
	T &\rightarrow \{J: E/T,\sigma,\sigma_{1},\sigma_{2},\sigma_{3}\}/\sim.
	\nonumber
\end{align}
Here ($J:E/T,\sigma$) is a smooth proper family of elliptic curves, namely $J$ is a smooth proper family of irreducible curves of genus one and $\sigma$ is a section of $J$. Moreover each $\sigma_{j}$ is a section of $J$ such that 
\begin{equation}
 \CO_{E}(\sum_{j=1}^{3}\sigma_{j}(T)) \cong \CO_{E}(3\sigma(T)) \otimes J^{*}(\CL),
 \end{equation}
  for some line bundle $\CL$ on $T$, and $\sigma(p) \not= \sigma_{j}(p)$ for any geometric point $p$ of $T$ and for any $j=1,2,3$. 
\end{defn}
\begin{prop}
	$\CM'$ admits a coarse modul space as an open subset of $\BP(1,2,2,3)$ whose complement is a hypersurface of degree 12.  More precisely, given a datum $(E,P,P_{i},i=1,2,3)$ in $P\CM'(Spec(\BC))$. Write it in normal form as
		\begin{equation}
		y^{2}=x^{3}+a_{4}x+a_{6},
	\end{equation}
	where $P=(0,1)$, and we write $P_{1}=(a_{2},a_{3}),P_{2}=(a_{2}',a_{3}'),P_{3} =(a_{2}'',a_{3}'')$. Let $a_{1}$ be the slope of the line connecting $P_{1}$ and $P_{2}$. Then the correponding point in $\BP(1,2,2,3)$ is $[a_{1}:a_{2}:a_{2}':a_{3}]$.  Moreover the following equality holds:
	\begin{align}
		a_{3}'&=a_{3}+a_{1}(a_{2}'-a_{2}), a_{2}''=a_{1}^{2}-a_{2}-a_{2}', a_{3}''=a_{1}^{3}-3(a_{1}a_{2}-a_{3})-a_{3}-a_{3}'; \\ \nonumber	a_{4}&=a_{1}(a_{3}+a_{3}')-(a_{2}^{2}+a_{2}a_{2}'+a_{2}'^{2}), 
		a_{6}=a^{2}_{3}-a_{1}a_{2}(a_{3}+a_{3}')+a_{2}a_{2}'(a_{2} + a_{2}'),\nonumber
		\end{align}
and we have $\Delta(E):=4a_{4}^{3}+27a_{6}^{2} \not=0$.
\end{prop}
\begin{proof}
We write $\BP(1,2,2,3)=Proj(\BC[c_{1},c_{2},c_{2}',c_{3}])$, where $c_{i},c_{j}'$ are generators of degree $i,j$ respectively. Denote  $c_{3}',c_{4},$ and $c_{6}$ by the same formula as (A.3) in terms of $c_{1},c_{2},c_{2}'$ and $c_{3}$. Let $M'=D_{+}(4c_{4}^{3}+27c_{6}^{2}) \subseteq \BP(1,2,2,3)$. Then we aim to construct a natural transformation $\xi':  \CM' \rightarrow h_{M'}$ which makes $M'$ to be the coarse moduli space of $\CM'$.  To do so, we are going to write a family of elliptic curves (Zariski locally) in Weierstrass forms. More  precisely we recall the following standard fact:
	\begin{lem}
	Let $(J:E/T,\sigma)$ be a family of elliptic curves over $T$. Denote by $D=\sigma(T)$ the effective Cartier divisor of $E$ and denote by $\CL=\CO_{E}(D)$ the associated line bundle on $E$.  Then:\par 
	(1) There exists an open covering of $T=\cup_{i}T_{i}$ such that $J_{*}(\CL^{\otimes{k}})$ and $J_{*}(\CL^{\otimes{k+1}})/J_{*}(\CL^{\otimes{k}})$ are free of rank $k$ and $1$ respectively on each $T_{i}$ and for any $k \leqslant 6$.\par 
	(2) For each $T_{i}$ there exists a choice of a basis of $J_{*}(\CL^{\otimes{3}})(T_{i})$ namely $1, x_{i},y_{i}$ such that $x_{i} \in J_{*}(\CL^{\otimes{2}})(T_{i}) \backslash J_{*}(\CL)(T_{i})$, $y_{i} \in J_{*}(\CL^{\otimes{3}})(T_{i}) \backslash J_{*}(\CL^{\otimes{2}})(T_{i})$ and
	satsfying:
	\begin{equation}
		y_{i}^{2}=x_{i}^{3}+a_{4,i}x_{i}+a_{6,i},
	\end{equation}
for some $a_{4,i},a_{6,i} \in \CO(T_{i})$.\par 
  (3) For each $T_{i}$ the sections $1,x_{i},y_{i} \in \CL^{\otimes{3}}(J^{-1}(T_{i}))$ induces a closed embedding of $J^{-1}(T_{i})$ into $\BP^{2}_{T_{i}}$ whose image is $V_{+}(y_{i}^{2}z_{i}-(x_{i}^{3}+a_{4,i}x_{i}z_{i}^{2}+a_{6,i}z_{i}^{3}))$. Moreover $\sigma|_{T_{i}}=[0:1:0]$ and $E\cap J^{-1}(T_{i})$ equlas to the hyperplane section $V_{+}(z_{i}) \cap J^{-1}(T_{i})$.\par 
  (4) For any $i,i' \in I$, we have $x_{i}|_{T_{i} \cap T_{j}}=n_{i} x_{j}|_{T_{i} \cap T_{j}}$ and $y_{i}|_{T_{i} \cap T_{j}}=m_{i} y_{j}|_{T_{i} \cap T_{j}}$ for some $m_{i},n_{i} \in \CO(T_{i} \cap T_{j})^{*}$ such that $m_{i}^{2}=n_{i}^{3}$.
	\end{lem}
Now for a given a datum $(J:E/T,\sigma,\sigma_{1}, \sigma_{2},\sigma_{3})$ representing a class in $P \CM'(T)$,  we can choose $T=\cup_{i}T_{i}$ and suitable sections $x_{i},y_{i}$ satisfying the properties in Lemma A.3. From Definition A.1 the image of $\sigma_{j}|_{T_{i}}$ lies in $D_{+}(z_{i})$. Thus we can write $\sigma_{1}|_{T_{i}}=[a_{2,i}:a_{3,i}:1]$, $\sigma_{2}|_{T_{i}}=[a_{2,i}':a_{3,i}':1]$ and $\sigma_{3}|_{T_{i}}=[a_{3,i}'':a_{3,i}'':1]$ for some $a_{2,i},a_{2,i}',a_{3,i},a_{3,i}',a_{2,i}'',a_{3,i}'' \in \CO(T_{i})$. By (A.1) and Lemma A.3(3) we conclude that $\sigma_{1}(T_{i})+\sigma_{2}(T_{i})+\sigma_{3}(T_{i})=H_{i} \cap J^{-1}(T_{i})$ for some hyperplane $H_{i}=V_{+}(\alpha_{i}x_{i}+\beta_{i}y_{i}+\gamma_{i}z_{i}) \subseteq \BP_{T_{i}}^{2}$. As $\sigma_{j}(p) \not= \sigma(p)=[0:1:0]$ for any $j$ and for any geometric point $p$ on $T_{i}$, we conclude that $\beta_{i}$ is invertible in $\CO(T_{i})$. And for any point $p \in T_{i}$,  the slope of the lines connecting $\sigma_{1}(p),\sigma_{2}(p), \sigma_{3}(p)$ in $\BP_{k(p)}^{2}$ is given by $(\beta_{i}^{-1}\alpha_{i})(p)$. Denote by $a_{1,i}=\beta_{i}^{-1}\alpha_{i}$. Then the equation (A.3) holds for $a_{3,i}',a_{3,i}'',a_{2,i}'',a_{4,i}$ and $a_{6,i}$ in terms of $a_{1,i},a_{2,i},a_{2,i}',a_{3,i}$. To show this, it suffices to check on any geometric point $t \in T_{i}$, which follows from \cite[Section 2.4]{BH}. Here we note that there is a flaw in the expression of $a_{2}''$ in \cite[Section 2.4]{BH}: the coefficient of the term $a_{1}^{2}$ should be 1 instead of 3, which can be checked directly.
Next  we recall the following fact:
\begin{lem}
	Let $R$ be a ring, then the closed subscheme of $\BP_{R}^{2}$ defined by the Weierstrass equation
	\begin{equation*}
		y^{2}z=x^{3}+b_{4}xz^{2}+b_{6}z^{3}
	\end{equation*}
	is an elliptic curve over $R$ if and only if $\Delta=4b_{4}^{3}+27b_{6}^{2}$ is invertible in $R$.
\end{lem}
Thus we conclude that $\Delta_{i}=4a_{4,i}^{3}+27a_{6,i}^{2}$ is invertible in $\CO(T_{i})$ for each $T_{i}$. Now for each $T_{i}$ we can assign it a morphism $f_{i}$ to $M'=D_{+}(4c_{4}^{3}+27c_{6}^{2})$ by
\begin{align*}
	 f_{i}(t)=[a_{1,i}(t):a_{2,i}(t):a_{2,i}'(t):a_{3,i}(t)],
	 \end{align*}
 for any point $t \in T_{i}$.
Then by Lemma A.3(4) we conclude that $\{f_{i}:T_{i} \rightarrow M', i \in I\}$
glues to a morphism $f: T \rightarrow M'$ which does not depend on the choice of the open covering and the choice of sections $x_{i},y_{i}$ as well as the choice of the representative of a given class. Furthermore one can easily check that this induces a natural transformation $\CM' \rightarrow h_{M'}$. On the other hand, denote by $\widetilde{M}'=\BC^{4}\backslash V(\Delta)$ where $\Delta$ is defined as in (A.3). Then one sees that there exists a natural datum $\alpha' \in \CM'(\widetilde{M}')$ associated to $\widetilde{M}'$ given as above. And the induced morphism from $\widetilde{M}'$ to $M'$ is given by the GIT quotient of the correponding $\BC^{*}$-action.
 Then using the datum $\alpha'$ one can easily verify the universal property of coarse moduli spaces. Finally the fact that $\xi'(Spec(\BC))$ is bijective is clear from our construction.
\end{proof}
Next we relate $M'$ to the space $M$ by constructing a natural datum in $P\CM'(\CU_{D_{4}})$.
 Recall from diagram (3.2) that there is a morphism $\kappa_{D_{4}}:\CC_{D_{4}} \rightarrow \BP^{1} \times \BP^{1} \times \BP^{1}$ and there is a line bundle of degree two $\CL_{D_{4}}$ on $\CC_{D_{4}}/\CU_{D_{4}}$. Moreover there are three line bundles of degree two on $\CC_{D_{4}}/\CU_{D_{4}}$ defined by $\CL_{i}=(p_{i} \circ \kappa_{D_{4}})^{*}(\CO_{\BP^{1}}(1))$ for $i=1,2,3$. By Lemma 3.9, the sheaf $Pic_{(\CC_{G}/\CU_{G})(\text{\'{e}tale})}^{2}$  is represented by a scheme $Pic^{2}(\CC_{D_{4}}/\CU_{D_{4}})$ which is a smooth projective family of curves of genus one over $\CU_{D_{4}}$. Thus the line bundles $\CL_{D_{4}},\CL_{1},\CL_{2},\CL_{3}$ are given by sections of $Pic^{2}(\CC_{D_{4}}/\CU_{D_{4}})/\CU_{D_{4}}$ which we denote by $\sigma,\sigma_{1},\sigma_{2}$ and $\sigma_{3}$ respectively. Then $(Pic^{2}(\CC_{D_{4}}/\CU_{D_{4}})/\CU_{D_{4}},\sigma)$ is a family of elliptic curves and from Proposition 1.6(2) we conclude that the datum $(Pic^{2}(\CC_{D_{4}}/\CU_{D_{4}})/\CU_{D_{4}},\sigma,\sigma_{i},i=1,2,3)$ defines a datum $\beta_{D_{4}} \in \CM'(\CU_{D_{4}})$.\par 
 The open embedding of $M$ into $\BP(1,3,4,6)$ is given as follows:
\begin{prop}\label{para_}
	 Let $\xi'(\CU_{D_{4}})(\beta_{D_{4}}): \CU_{D_{4}} \rightarrow M'$ be the morphism induced by $\beta_{D_{4}}$ then:\par 
	(1).  $\xi'(\CU_{D_{4}})(\beta_{D_{4}})$ is $H_{D_{4}}$-invariant and the induced quotient map $M_{D_{4}} \rightarrow M'$ is an isomorphism.
	\par 
	(2). The $S_{3}$-action on $M_{D_{4}} \cong M'$ induces a natural $S_{3}$-action on $\BP(1,2,2,3)$. The induced quotient map gives an open embedding of $M=M_{D_{4}}//S_{3}$ into $\BP(1,2,2,3)//S_{3} \cong \BP(1,3,4,6)$ whose complement is a hypersurface of degree 12.
\end{prop}
\begin{proof}
That  $\xi'(\CU_{D_{4}})(\beta_{D_{4}})$ is $H_{D_{4}}$-invariant follows from the fact that $H_{D_{4}}$ naturally acts on the family $\CC_{D_{4}}/U_{D_{4}}$ and the line bundles $\CL,\CL_{1},\CL_{2},\CL_{3}$. As both $M_{D_{4}}$ and $M'$ are normal varieties, it suffices to show that the quotient map is bijective. By Proposition 1.6(1) and Corollary 1.9(1), $M_{D_{4}}$ parametrizes isomorphic classes of data consisting of $(C,L,L_{i},i=1,2,3)$. And the quotient morphism maps an isomorphic class of $(C,L,L_{i},i=1,2,3)$ to the class of $(Jac^{2}(C),L,L_{i},i=1,2,3)$. Then one can easily verify the bijectivity of the quotient map by choosing an isomrophism between a given curve $C$ of genus one and its Jacobian  $Jac^{2}(C)$. \par 
	To see (2), note that for a given class of datum $(E,P,P_{1},P_{2},P_{3})$ as in Proposition A.2, the induced $S_{3}$-action on the datum is given by the corresponding permutation on $P_{1},P_{2}$ and $P_{3}$. Then following the descripition in Proposition A.2, one sees that the $S_{3}$-action is induced by an $S_{3}$-action on $\BP(1,2,2,3)$ defined by:
	\begin{align*}
	(12).c_{1}&=c_{1}, (12).c_{2}=c_{2}', (12).c_{2}'=c_{2}, (12).c_{3}=c_{3}+c_{1}(c_{2}'-c_{2});\\
	(123).c_{1}&=c_{1}, (123).c_{2}=c_{2}', (123).c_{2}'=c_{1}^{2}-c_{2}-c_{2}', (123).c_{3}=c_{3}+c_{1}(c_{2}'-c_{2}).
	\end{align*}
Then a direct calculation yields that the invariant  subring $\BC[c_{1},c_{2},c_{2}',c_{3}]^{S_{3}}=\BC[d_{1},d_{3},d_{4},d_{6}]$ is a polynomial ring in four generators $d_{1},d_{3},d_{4},d_{6}$ of degree 1,3,4,6 respectively. More precisely:
\begin{align}
	d_{1}&=c_{1}, \\ \nonumber 
	d_{3}&=c_{1}c_{2}-c_{3},\\ \nonumber d_{4}&=c_{2}^{2}+c_{2}'^{2}+c_{2}c_{2}'-c_{1}^{2}(c_{2}+c_{2}'),\\ \nonumber
	 d_{6}&=c_{2}c_{2}'(c_{2}+c_{2}'-c_{1}^{2}) \nonumber.
	\end{align}
The above construction induces an open embedding of $M=M_{D_{4}}//S_{3}$ into $\BP(1,2,2,3)//S_{3} \cong \BP(1,3,4,6)$.
Note that the $S_{3}$-action on a given datum $(C,L,L_{1},L_{2},L_{3})$ fixes the defining equation of the elliptic curve, whence it fixes $\Delta=4c_{4}^{3}+27c_{6}^{2}$. This implies that the image of the open embedding of $M$ into $\BP(1,3,4,6)$ equals to the principle open subset $D_{+}(\Delta)$.
\end{proof}
\section*{Acknowledgment}
I am grateful to Thomas Dedieu and Laurent Manivel for kindly introducing me this problem, and for their helpful discussions. I am particularly indebted to Laurent Manivel for valuable suggestions on this article, and for providing me the important reference \cite{BH}.  I would like to thank Baohua Fu for his constant support and for his suggestions on this article. 
Most part of this article was written during my visit at Toulouse Mathematics Institute in 2023, supported by a scholarship from the University of Chinese Academy of Sciences. I would like to thank Toulouse Mathematics Institute for its hospitality.
	
\end{document}